\documentclass{amsart}

\usepackage{graphicx}
\usepackage[utf8]{inputenc}
\usepackage[english]{babel}
\usepackage{csquotes}
\usepackage{amssymb}
\usepackage{tikz-cd}

\usepackage{cite}
\usepackage{amsmath,amsfonts,pictex,graphicx,fullpage,etex,tikz,hyperref}
\usepackage{enumerate}
\usepackage{enumitem}
\newtheorem{theorem}{Theorem}[section]
\newtheorem{lemma}[theorem]{Lemma}
\newtheorem{proposition}[theorem]{Proposition}
\newtheorem{corollary}[theorem]{Corollary}

\theoremstyle{definition}
\newtheorem{definition}[theorem]{Definition}
\newtheorem{problem}[theorem]{Problem}

\theoremstyle{remark}
\newtheorem{remark}[theorem]{Remark}
\newtheorem{notation}[theorem]{Notations}

\numberwithin{equation}{section}
\makeatletter

\newcommand{\T}{\mathrm{T}}

\newcommand{\R}{\mathbb{R}}
\newcommand{\Z}{\mathbb{Z}}

\newcommand{\x}{\mathbf{x}}
\newcommand{\y}{\mathbf{y}}

\newcommand{\dd}{\mathrm{d}}
\makeatother

\begin{document}

\title{Hausdorff dimension of divergent diagonal geodesics on product of finite volume hyperbolic spaces}

\author{Lei Yang $^{\ast}$}
\address{Mathematical Sciences Research Institute, Berkeley, CA, 94720, U.S.A.}

\curraddr{Einstein Institute of Mathematics, Hebrew University of Jerusalem, Jerusalem, 9190401, Israel}
\email{yang.lei@mail.huji.ac.il}
\thanks{$^{\ast }$ Supported in part by a Postdoctoral Fellowship at MSRI}


\date{}


\begin{abstract}
In this article, we consider the product space of several non-compact finite volume hyperbolic spaces, 
$V_1, V_2, \dots , V_k$ of dimension $n$. Let 
$\mathrm{T}^1(V_i)$ denote the unit tangent bundle of $V_i$ for each $i=1,\dots , k$, then for every 
$(v_1, \dots , v_k) \in \mathrm{T}^1 (V_1) \times \cdots \times \mathrm{T}^1 (V_k)$, the diagonal geodesic flow $g_t$ is defined by
$g_t (v_1, \dots , v_k) = (g_t v_1, \dots , g_t v_k)$. And we define 
$$\mathfrak{D}_k =\left\{ (v_1, \dots, v_k) \in \mathrm{T}^1 (V_1) \times \cdots \times \mathrm{T}^1 (V_k): g_t(v_1, \dots, v_k) 
\text{ divergent, as } t\rightarrow \infty\right\}. $$
\par We will prove that the Hausdorff dimension of $\mathfrak{D}_k$ is equal to $k(2n-1) - \frac{n-1}{2}$. This extends the result of 
Yitwah Cheung ~\cite{Cheung1}.
\end{abstract}

\maketitle
\section{Introduction}
\par In ~\cite{Cheung1}, Yitwah Cheung considers the following interesting problem.
\par Let $\mathcal{M}_k$ be the product space of $k$ copies of $\mathrm{SL}(2,\Z)\setminus\mathrm{SL}(2,\R)$, 
$$\mathcal{M}_k := \left(\mathrm{SL}(2,\Z)\setminus \mathrm{SL}(2,\R)\right)^k.$$
Let $A$ denote the diagonal subgroup of $\mathrm{SL}(2,\R)$:
$$A := \left\{a(t)= \begin{bmatrix}e^t & \\ & e^{-t}\end{bmatrix}: t \in \R \right\}.$$
The diagonal action of $A$ on $\mathcal{M}_k$ is defined as follows: for $(v_1, v_2, \dots, v_k) \in \mathcal{M}_k$, 
where $v_i \in \mathrm{SL}(2,\Z)\setminus \mathrm{SL}(2,\R)$ for $i=1,2,\dots, k$, 
$$(v_1, v_2,\dots,v_k)a(t) : = (v_1 a(t),v_2 a(t),\dots, v_k a(t)) \in \mathcal{M}_k.$$
The divergent set $\mathcal{D}_k$ is defined to be collection of points $(v_1,v_2,\dots, v_k)\in \mathcal{M}_k$ such that 
$a(t)(v_1,v_2,\dots,v_k)$ diverges as $t\rightarrow +\infty$. One could ask what the Hausdorff dimension of $\mathcal{D}_k$ 
is.
\par From geometric point of view, $\mathrm{SL}(2,\Z)\setminus \mathrm{SL}(2,\R)$ can be identified as the unit tangent 
bundle of the hyperbolic space $\mathrm{SL}(2,\Z)\setminus \mathbb{H}^2$. The action of $A=\{a(t): t\in \R\}$ is the geodesic flow 
$\{g_t: t \in \R\}$ on $\T^1(\mathrm{SL}(2,\Z)\setminus \mathbb{H}^2)$. So $\mathcal{M}_k$ can be regarded as product of $k$ copies of unit tangent bundle 
of hyperbolic space $\mathrm{SL}(2,\Z)\setminus \mathbb{H}^2$, and the diagonal action of $A^{+}$ is the 
diagonal geodesic flow $g_t\times g_t \times \cdots \times g_t$ on $\mathcal{M}_k$. $\mathcal{D}_k$ is regarded as 
$$\mathcal{D}_k := \{(v_1,v_2,\dots, v_k)\in \mathcal{M}_k: (g_t v_1, g_t v_2,\dots, g_t v_k )\rightarrow \infty \text{ as } t \rightarrow +\infty\}.$$
\par When $k=1$, the problem is easy. According to a result of Dani ~\cite{Dani1}, for $v\in \mathrm{SL}(2,\Z)\setminus \mathrm{SL}(2,\R)$,
$a(t)v$ diverges as $t \rightarrow +\infty$ if and only if $v$ belongs to some closed $U$-orbit 
in $\mathrm{SL}(2,\Z)\setminus \mathrm{SL}(2,\R)$, where 
$$U:= \left\{u(x)= \begin{bmatrix}1 & x \\ 0 & 1\end{bmatrix}: x\in\R\right\}$$
denotes the horocyclic subgroup contracted by $A^{+} := \{a(t): t > 0\}$. Therefore the Hausdorff dimension is equal to 
$\dim \mathcal{M}_1 -1 = 2$.
\par When $k \geq 2$, the problem becomes interesting and complicated. It turns out that most divergent trajectories have non-divergent 
projection to each component, that is to say, most $(v_1,v_2,\dots, v_k)\in \mathcal{D}_k$ satisfies that for each $i=1,2,\dots, k$, 
$\{a(t)v_i: t >0\}$ does not diverge as $t\rightarrow +\infty$. In ~\cite{Cheung1}, Cheung showed that the Hausdorff dimension of 
$\mathcal{D}_k$, 
$$\dim_H \mathcal{D}_k = \dim \mathcal{M}_k - \frac{1}{2}.$$
In ~\cite{Cheung1}, Cheung established some general strategy to get lower bound and upper bound of Hausdorff dimension. To compute the Hausdorff 
dimension of $\mathcal{D}_k$, Cheung made use of continued fractions to encode trajectories under the action of $A$, and a result on counting 
integer points in a particular region of $\R^2$, which he proved in ~\cite{Cheung2}.
\par Because of this interesting result, it is natural to ask what happens if we replace $\mathbb{H}^2$ by $\mathbb{H}^n$ and replace the special lattice 
$\mathrm{SL}(2,\Z)$ by other noncocompact lattices $\Gamma_i < \mathrm{Iso}(\mathbb{H}^n) = \mathrm{SO}(n,1)$. To be precise, one could consider $k$ 
noncompact hyperbolic spaces
$V_i =\Gamma_i\setminus \mathbb{H}^n$ where $\Gamma_i < \mathrm{SO}(n,1)$ is a noncocompact lattice of $\mathrm{SO}(n,1)$, i.e., $V_i$ is not 
compact and has finite volume. We define $\mathcal{M}_k := \T^1(V_1)\times \T^1(V_2) \times \cdots \times \T^1(V_k)$, and consider
the diagonal geodesic flow $g_t\times g_t \times \cdots \times g_t$ on $\mathcal{M}_k$ defined the same as above. Let $\mathcal{D}_k$ denotes 
the set of points with divergent forward trajectories, one could ask what the Hausdorff dimension of $\mathcal{D}_k$ is.
\par In this article we extend Cheung's work as follows:
\begin{theorem}
 \label{goal}
 Let $k \geq 2$ and $V_1, V_2, \dots, V_k$ be $k$ non-compact finite volume hyperbolic spaces of dimension $n$ and 
 $$\mathcal{M}_k := \mathrm{T}^1(V_1)\times \cdots \times \mathrm{T}^1(V_k)$$
 with diagonal geodesic flow $g_t : \mathcal{M}_k \rightarrow \mathcal{M}_k$ described as before. Denote 
 $$\mathcal{D}_k:= \{\mathfrak{m} \in \mathcal{M}_k : g_t(\mathfrak{m}_k) \text{ diverges as } t \rightarrow \infty \},$$
 then its Hausdorff dimension $\dim_H \mathcal{D}_k = k(2n-1) -\frac{n-1}{2}$.
\end{theorem}
\par This work can be regarded as a part of a large program of studying 
the behavior of trajectories under diagonal flow and calculation of Hausdorff dimension
of trajectories with certain properties in various dynamical systems. In ~\cite{Kleinbock_Margulis}, 
Kleinbock and Margulis studied bounded trajectories on homogeneous spaces 
under nonquasiunipotent flows and proved that the set of bounded trajectories has
full Hausdorff dimension, although it has zero Lebesgue measure. The calculation makes use of the 
mixing property of nonquasiunipotent flows which is also the main tool of this article. As per
divergent trajectories, Yitwah Cheung studied the trajectories on the homogeneous space 
$\mathrm{SL}(3,\mathbb{R})/ \mathrm{SL}(3,\mathbb{Z})$ under the diagonal flow 
$a(t) := \mathrm{diag}\{e^t, e^t, e^{-2t} \}$
and proved that the Hausdorff dimension of set of divergent trajectories is 
$\dim (\mathrm{SL}(3,\mathbb{R})/ \mathrm{SL}(3,\mathbb{Z})) - \frac{2}{3}$, i.e., the Hausdorff co-dimension
is equal to $\frac{2}{3}$ (see ~\cite{Cheung3}). This result was 
recently extended to the space 
$\mathrm{SL}(d+1, \mathbb{R}) / \mathrm{SL}(d+1, \mathbb{Z})$ with diagonal flow 
$\mathrm{diag}\{e^t, e^t,\dots, e^t, e^{-dt}\}$ for arbitrary
$d \geq 2$ by Cheung and Chevallier (see ~\cite{Cheung_Chevallier}).  If we do not restrict our
attention on homogeneous spaces, we will find that the geodesic flows on translation surfaces share 
many properties in common. In ~\cite{Cheung2}, Cheung showed that the Hausdorff dimension of nonergodic
direnctions of some particular translation surface is equal to $1/2$, with main ideas similar to that 
of ~\cite{Cheung1}. This work was later extended by Cheung, P. Hubert and H. Masur 
in ~\cite{Cheung_Hubert_Masur}. The work in this direction can date back to the work of Masur and Smillie (see ~\cite{Masur_Smillie}) and
that of Masur (see ~\cite{Masur}).
\par The basic idea to compute $\dim_H \mathcal{D}_k$ goes as follows: 
\par For a non-compact finite volume $n$-dimensional hyperbolic space $V \cong \Gamma \setminus \mathbb{H}^n$, we focus our attention to
the set of cusp points with respect to $\Gamma$ on the ideal boundary sphere 
$\partial \mathbb{H}^n \cong \mathbb{S}^{n-1}$ of $\mathbb{H}^n$, each of which
corresponds to an infinite end of a fundamental domain of $V$. To study a particular geodesic ray $\mathcal{G}$ in $V$, 
it suffices to consider one of its lifts $\tilde{\mathcal{G}}$ in $\mathbb{H}^n$,
then at some moment $\mathcal{G}$ is ''near $\infty$ " in $V$ if and only if $\tilde{\mathcal{G}}$ is ``close to" 
some cusp point on $\partial \mathbb{H}^n$, this associates every geodesic ray with a sequence of cusp points, 
and the time $\tilde{\mathcal{G}}$ stays near a cusp point can be estimated by some quantity called the 
height of the cusp point.
\par To get the lower bound of the Hausdorff dimension of $\mathcal{D}_k$, it suffices to consider the case when $k=2$ 
(since if a trajectory has divergent projection on the first two components, then itself is also divergent). On the first component, we choose the 
selection of geodesic rays such that the height of next cusp point is ``much larger" than the preceding one, namely, the associate sequence of 
cusp points $\{\mathfrak{a}_k : k \in \mathbb{N}\}$ satisfies that for any $k \in \mathbb{N}$,
$h(\mathfrak{a}_{k+1}) \asymp h^{1+\delta}(\mathfrak{a}_k)$, here $h(\cdot)$
denotes the height of the cusp point, and $\delta > 0$ denotes some small constant. For each such geodesic ray $\mathcal{G}$ associated 
with $\{\mathfrak{a}_k: k \in \mathbb{N}\} $on the 
first component, we choose the set of geodesics on the second component whose heights of corresponding cusp points all stay ``far away" from 
each $h(\mathfrak{a}_k)$, i.e., no height lies in the interval 
$[\frac{h(\mathfrak{a}_k)}{\log h(\mathfrak{a}_k)}, h(\mathfrak{a}_k)\log h(\mathfrak{a}_k)]$ for all $k \in \mathbb{N}$. 
Every pair chosen as above gives a divergent trajectory on the product space, and the choice of 
the first component gives Hausdorff dimension $\frac{n-1}{2}$, and the second component contributes full Hausdorff dimension, this gives the lower 
bound of the Hausdorff dimension. 
\par As per the upper bound of the Hausdorff dimension, we firstly choose a compact subset 
$\mathcal{K}_{\rho} \subset \mathcal{M}_k$ depending on some parameter $\rho > 0$, and define
$$
E_k(\rho) := \{\mathfrak{m} \in \mathcal{M}_k : g_t(\mathfrak{m}) \not\in \mathcal{K}_{\rho} \text{ for all large }t\} ,
$$
and then construct a so called self-similar covering of $E_k(\rho)$, and then apply the inequality proved in ~\cite{Cheung1} to find
the upper bound. This upper should depend on the parameter $\rho$, by letting $\rho \rightarrow 0$, we will get the same upper bound as the 
lower bound. 
\par Compared with the work of Cheung (~\cite{Cheung1}), the new ingredients of this article are as follows:
\par At first, in the work of Cheung, one only
considers the special hyperbolic surface $\mathrm{SL}(2, \Z) \setminus \mathrm{SL}(2,\R)$. In this case, the set of cusp points is 
the set of rational points (including $\infty$) on $\partial \mathbb{H}^2 \cong \R \cup \{\infty\}$, and the height of a rational point
$\frac{p}{q}$ ($p$ and $q>0$ are coprime) is naturally its denominator $q$. But in general case, one needs to define the height of
a cusp point properly so that it has most of the nice properties of the denominator of a rational number. This work is done in the third 
section of the article.
\par Secondly, as we have mentioned above,
to get the lower bound of the Hausdorff dimension of $\mathcal{D}_k$, one needs to fix a geodesic ray $\mathcal{G}$ 
in the first component with property given above, with this fixed 
geodesic ray, one needs to choose the geodesic rays on the second component with heights far away from the heights of cusp points of 
$\mathcal{G}$. To compute the Hausdorff dimension of the geodesic rays on the second component, we will deal with the following counting 
problem on cusp points:
\begin{problem}
 \label{counting problem}
 Given a cusp point $\mathfrak{a} \in \R^{n-1}$, and some large number $t > 0$, how many cusp points $\mathfrak{b}$ satisfy that 
 $h(\mathfrak{b}) \in [e^t h(\mathfrak{a}), 2 e^t h(\mathfrak{a})]$ and $\|\mathfrak{b} - \mathfrak{a}\| \leq \frac{1}{h(\mathfrak{a})}$?
 Here $\|\cdot\|$ denotes the Euclidean norm on $\R^{n-1}$.
\end{problem}

\par In ~\cite{Cheung1}, the above counting problem is tailored to the following counting problems on rational numbers:
\begin{problem}
 Given $d >0$ small and $h >0$ large, such that $hd$ is small enough but $h^2 d$ is large enough, $x \in \R$ is some real number such that 
 there is a convergent $\frac{p}{q}$ of $x$ satisfying $(hd)^{-1}\leq q \leq h$, then how many reduced rational numbers $\frac{p'}{q'}$ in the
 interval $[x-d, x+d]$ satisfy that $q' \in [h, 2h]$?
\end{problem}
This question was answered in another work ~\cite{Cheung2}. In that work, the counting problem on rational numbers was reduced to counting problem
on integer points in $\Z^2$ inside a particular region of $\R^2$. The counting was done via careful study of integer points in $\R^2$.
\par But this argument could not be modified to solve the above general counting problem, the correspondence between cusp points in 
$\R$ and lattice points in $\R^2$ only exists for $\Gamma = \mathrm{SL}(2, \Z)$. For general case, a new approach is needed.
\par It turns out that by making use of the mixing property of geodesic flow on $\T^1(V)$, the counting 
can be done. The detail will be discussed in the fourth section.
\par The result of this article is possible to extend in the following directions:
\begin{problem}
\par One can drop the finite volume condition of the hyperbolic spaces, and instead,
assume that every component $V_i$ is geometrically finite. In this case, one can define $\mathcal{D}_k$ as follows:
$$
\mathcal{D}_k : = \left\{(v_1, \dots, v_k) \in \mathcal{M}_k : g_t(v_1, \dots, v_k) \text{ diverges but } g_t(v_i) \text{ does not diverges for each } 
i \right\}.
$$
And ask what is the Hausdorff dimension of $\mathcal{D}_k$.
\end{problem}
\begin{remark}
We add the condition that $g_t(v_i)$ does not diverge for each $i$ because if $V_i$ has 
infinite volume, the set of divergent trajectories on each $V_i$ has 
full Hausdorff dimension (actually, it has positive Lebesgue measure), the problem will be trivial without the additional condition.
\end{remark}
\par The article is organized as follows: 
\begin{itemize}
\item In the second section, we will recall some basic theory of Lie groups and hyperbolic spaces, and make a basic 
reduction of the original problem.
\item In the third section, we discuss the structure of a general finite volume hyperbolic space $V =\Gamma\setminus \mathbb{H}^n$, general 
properties of cusp points on the ideal infinite boundary $\partial \mathbb{H}^n$
and basic properties of geodesic rays on $V$.
\item In the fourth section, we will prove the counting result on cusp points mentioned above. This is the most important technical result for 
getting the lower bound of Hausdorff dimension of $\mathcal{D}_k$.
\item In the fifth section, we will finish the computation of Hausdorff dimension, the first part will give the lower bound, and the second part 
will give the upper bound.
\end{itemize}
\begin{notation}
\par We will use the following notations: for two quantities $A$ and $B$, we will use $A \ll B$ 
to mean that there is a constant $C>0$, only depending on
the structure of $\Gamma_i$'s, such that $A \leq CB$, we use $A \gg B$ to mean that $B \ll A$, and 
use $A \asymp B$ to mean that $A \ll B$ and $B \ll A$.
\end{notation}

\noindent {\bf Acknowledgement: }This article is part of my thesis, 
I would like to express my deep gratitude to my advisor, Professor Nimish Shah, for his immensurable amount of support 
and guidance during the process of this work. I also would like to thank Dmitry Kleinbock and Yitwah Cheung for reading an earlier
version of this paper and giving a lot of comments and suggestions. Thanks are also due to the referees for many useful suggestions.
\section{Prelimenaries on hyperbolic spaces and basic reduction}
In this section, we recall some basic theory of Lie groups and hyperbolic spaces, and reduce the original problem to 
a relatively simple problem.
\par Let $V$ be some non-compact hyperbolic space of dimension $n$ with finite total volume, then we have 
$V= \Gamma \setminus \mathbb{H}^n$, where $\mathbb{H}^n$ is the 
universal $n$-dimensional hyperbolic space and $\Gamma = \pi_1(V)$. It is well known that 
$\mathrm{Iso}(\mathbb{H}^n) \cong \mathrm{SO}(n,1)$, in this article, we denote it by $G$. For $x \in \mathbb{H}^n$, 
the group of stabilizers of $x$ in $G$ is $K \cong \mathrm{SO}(n)$, so $\mathbb{H}^n \cong G/K$, and 
for $v \in \T^1(\mathbb{H}^n)$, the group of stabilizers of $v$ in $G$ is $M \cong \mathrm{SO}(n-1)$, 
so the unit tangent bundle $\mathrm{T}^1 (\mathbb{H}^n) \cong G/M$. And $\Gamma$ can be identified with a discrete subgroup
of $G$, such that $\Gamma \setminus G$ admits a finite measure invariant under the right multiplication of $G$.
We denote the Lie algebra of $G$ by $\mathfrak{g} =\mathfrak{so}(n,1)$, according to the theory of Lie groups and Lie algebras, 
$G$ admits a one-dimensional maximal $\mathbb{R}$-split torus $A$ and a characteristic $\lambda : A \rightarrow \mathbb{R}_{+}$
such that $\mathfrak{g}$ decomposes as follows according to the adjoint action 
of $A$:
$$\mathfrak{g} = \mathfrak{g}_{-1} \oplus \mathfrak{z}(A) \oplus \mathfrak{g}_{+1}$$
where $\mathfrak{z}(A)$ is the Lie algebra of the centralizer $Z(A)\cong MA$ of $A$, and 
$$\mathfrak{g}_{\pm 1} = \{v \in \mathfrak{g} : \mathrm{Ad}(a) v = \lambda(a)^{\pm 1} v \text{ for any } a \in A\}.$$
We parametrize $A = \{ a(t): t \in \mathbb{R}\}$ such that $\lambda (a(t)) = e^t$.
\par The Weyl group element $\sigma$ with respect to the torus $A$ has a representative in $K$, which we also denote by $\sigma$. Then
$\sigma^2 = \mathrm{id}$ and $\sigma a(t) \sigma^{-1} = a(-t)$.
\par It is well known that $\mathfrak{g}_{+1} \cong \R^{n-1}$, and $\mathfrak{g}_{+}$ is the Lie algebra of the expanding horospherical 
subgroup $N$ with respect to the conjugate action of $A$. We identify $\mathfrak{g}_{+}$ by $\R^{n-1}$ and parametrize $N$ by
$$N = \left\{u(\x) = \exp(\x): \x \in \R^{n-1} \cong \mathfrak{g}_{+1} \right\}.$$
Similarly, $\mathfrak{g}_{-1} \cong \R^{n-1}$ is the Lie algebra of the contracting horospherical subgroup $U^{-}$ with respect to the 
conjugate action of $A$. Because $U^{-} =\sigma N \sigma$, we could parametrize $U^{-}$ by
$$U^{-} = \{u^{-}(\x) = \sigma u(\x) \sigma : \x \in \R^{n-1}\}.$$
In this article, we will fix a Euclidean norm $\|\cdot\|$ on $\R^{n-1}$. 
\par We have the following Iwasawa decomposition: 
$$ \begin{array}{l} 
N \times A \times K \rightarrow G \\ (n, a, k) \mapsto nak\end{array}$$
where the map is group multiplication, and it is a diffeomorphism.
\par It is also well known that 
$$
\begin{array}{l}
 N \times MA \times U^{-} \rightarrow G \\
 (u(\x), ma, u^{-}(\y)) \mapsto u(\x)ma u^{-}(\y)
\end{array}
$$
is a diffeomorphism.
\par Let $P = MAN$ denote a parabolic subgroup of $G$, we have the following
Bruhat decomposition:
$$G = P \cup N\sigma P$$
Then the ideal boundary $\partial \mathbb{H}^n \cong G/P = P/P \cup N\sigma P/P$, we may identify $N\sigma P/P$ with $\mathbb{R}^{n-1}$
and denote $P/P$ by $\infty$. Then we have $\partial \mathbb{H}^n = \mathbb{R}^{n-1} \cup \{ \infty\} \cong  \mathbb{S}^{n-1}$.
\par To study the homogeneous space $\Gamma \setminus G$, we need to know the shape of its 
fundamental domain, especially its shape near infinity. 
\par Let $\eta \subset N$ be a compact subset of $N$, and for some $s \in \mathbb{R}$, denote
$$A_s = \{a(t): t\geq s\} \subset A$$
We define
$$\Omega(\eta, s) = \eta A_s K$$
Thanks to Garland and Raghunathan, we have the following result concerning the fundamental domain:
\begin{theorem}{(See ~\cite[Theorem 0.6 and Theorem 0.7]{Garland_Raghrunathan})}
\label{G_R}
\par There exists $s_0 > 0$, a compact subset $\eta_0$ of $N$ and a finite subset $\Xi$ of $G$ such that
\begin{enumerate}
\item $G = \Gamma \Xi \Omega (s_0, \eta_0)$
\item for all $\xi \in \Xi$, the group $\Gamma \cap \xi N \xi^{-1}$ is a cocompact lattice in $\xi N \xi^{-1}$
\item for all compact subset $\eta$ of $N$ the set
$$\{\gamma \in \Gamma : \gamma \Xi \Omega (s_0, \eta) \cap \Omega (s_0, \eta) \neq \emptyset \}$$
is finite
\item for each compact subset $\eta$ of $N$ containing $\eta_0$ there exists $s_1 > s_0$ such that for all $\xi_1, \xi_2 \in \Xi$
and $\gamma \in \Gamma$ with $\gamma \xi_1 \Omega (s_0, \eta) \cap \xi_2 \Omega(s, \eta) \neq \emptyset$, we have $\xi_1 =\xi_2 $
and $\gamma \in \xi_1 N M \xi_1^{-1}$
\end{enumerate}
\end{theorem}
Without loss of generality, we may assume $e \in \Xi$, since otherwise we can replace $\Gamma$ with some conjugate $\xi^{-1} \Gamma \xi$ to move 
$\xi \in \Xi$ to $e$. Therefore we can define the set of cusp points of $\Gamma\setminus \mathbb{H}^n$ in 
the ideal boundary $\partial \mathbb{H}^n = \mathbb{R}^{n-1}\cup \{\infty\}$
to be $\Gamma \setminus \mathbb{H}^n$ to be $\Gamma \Xi \infty \subset \mathbb{R}^{n-1} \cup \{\infty\}$. Our 
assumption that $e \in \Xi$ is equivalent to saying that $\infty$ is a cusp of $\Gamma \setminus \mathbb{H}^n$. 
It is easy to see that $\Gamma x a(t)M$ diverges as
$t \rightarrow \infty$ if $x a(t)P/P \in \partial \mathbb{H}^n$ is a cusp of $\Gamma \setminus \mathbb{H}^n$.
\par Then for any $v \in \mathrm{T}^1 (V) \cong \Gamma \setminus G /M $, we can represent $v$ by $\Gamma x M$ for some $x \in G$, then the 
geodesic flow is just the group action of $A$, to be precise, $g_t (\Gamma x M) = \Gamma x a(t) M$. Note that $M \subset Z(A)$.
\par Then for a product of $k$ such spaces $\Gamma_1 \setminus G / M \times \cdots \times \Gamma_k \setminus G/M$, for the argument above,
we may assume that $\infty$ is a cusp of each $\Gamma_i \setminus \mathbb{H}^n$. Then for a general point 
$(\Gamma_1 x_1 M, \dots, \Gamma_k x_k M) \in \Gamma_1 \setminus G / M \times \cdots \times \Gamma_k \setminus G/M$, if any $\Gamma_i x_i$ has a 
representative in $P =MAN$,we may assume that $x_i\in P$ then $x_i a(t)P/P = \infty$, since $\infty$ is a cusp of $\Gamma_i \setminus \mathbb{H}$,
we have $\Gamma_i x_i a(t) M $ diverges in $\Gamma_i \setminus G / M$ as $t \rightarrow \infty$, 
and thus $(\Gamma_1 x_1 a(t)M, \cdots, \Gamma_k x_k a(t) M)$ 
diverges in the product space as $t \rightarrow \infty$. 
The set of such trajectories has Hausdorff codimension $n-1$. Now we assume every $x_i$ is of form $n_i \sigma p_i$, where $n_i \in N$ and 
$p_i \in P$, Then we have 
$(\Gamma_1 x_1 a(t) M, \dots, \Gamma_k x_k a(t) M) = (\Gamma_1 n_1 \sigma a(t) a(-t) p_1 a(t) M, \dots, \Gamma_k n_k \sigma a(t) a(-t) p_k a(t) M)$.
Since for any $n \in N$, $a(-t) n a(t) \rightarrow e$ as $t\rightarrow \infty$, and $MA \subset Z(A)$, we have for any large $t$, $a(-t)p_i a(t)$ remains in
some compact subset of $G$ depending on $p_i$, therefore, $(\Gamma_1 x_1 a(t) M, \dots, \Gamma_k x_k a(t) M)$ diverges if and only if 
$(\Gamma_1 n_1 \sigma a(t) M, \dots, \Gamma_k n_k \sigma a(t) M)$ diverges, as $t \rightarrow \infty$. Now we focus our attention 
to the geodesics of form $\{ u(\mathbf{x})\sigma a(t): t > 0 \}$ where $u(\mathbf{x}) \in N$ defined as above.
\par Define:
$$\mathfrak{B}_k =\{ (\mathbf{x}_1, \dots , \mathbf{x}_k) \in (\mathbb{R}^{n-1})^k: (\Gamma_1 u(\mathbf{x}_1)\sigma a(t) M, \dots, \Gamma_k u(\mathbf{x}_k)\sigma a(t) M) \text{ diverges as } t \rightarrow \infty \}$$
Then by the argument above and the property of Hausdorff dimension, we have
$$\dim_H \mathfrak{B}_k = \dim_H \mathfrak{D}_k - k \dim(P/M) $$
Since $\dim (P/M) = \dim (NA) = n$, we reduce the original problem to showing the following statement:
\begin{proposition}
\label{main result}
 $$\dim_H \mathfrak{B}_k = k(n-1) - \frac{n-1}{2}.$$
\end{proposition}
\section{Geodesic rays on hyperbolic space and cusp points}
In this section, we will fix a hyperbolic non-compact hyperbolic space $V = \Gamma \setminus \mathbb{H}^n$ with finite volume, and 
discuss properties of geodesic rays on $V$ and cusp points of $\Gamma$ on $\partial \mathbb{H}^n$. Because of the reduction in the 
previous section, we only look at geodesic rays of form $\{\Gamma u(\x) \sigma a(t) M: t >0 \}$.
\par By Theorem \ref{G_R}, a typical fundamental domain $\mathcal{F}$ of the 
action of $\Gamma$ on $\mathbb{H}^n$ is the union of a compact 
subset $\mathcal{K}$ and finitely infinite cusp of form $\gamma \xi_i \Omega(\eta, s_1)/K$, where
$\xi_i$ runs over elements of $\Xi$. We can choose a fixed fundamental domain $\mathcal{F}_0$
such that $K \in \mathcal{F}_0$, and we may change $\Xi$ so that $\mathcal{F}_0$ is union of 
a compact subset and $\xi_i \Omega(\eta, s_1)/K$.
\par By our discussion in the previous section, every cusp point is of form $\gamma \xi \infty$, 
where $\gamma \in \Gamma$, and $\xi \in \Xi$.
\begin{definition}
\label{def_height}
We consider the Bruhat decomposition of $\gamma \xi = u(\x) \sigma u(\y)a(r)m$ where $m \in M$, then we define the height of the cusp point 
$\mathfrak{a} = \gamma\xi \infty$ to be $h(\mathfrak{a}) = e^r$. It is easily seen that $\gamma \xi \in P$ iff and only $\gamma = \xi = e$, 
the corresponding cusp is $\infty$, we define the height of $\infty$ to be $h(\infty) = 1$.
\end{definition}
\begin{remark}
 For $m \in M \cong \mathrm{SO}(n-1)$, and $\mathbf{x} \in \R^{n-1}$, the conjugation
$m u(\x) m^{-2} = u(m \x)$, where the action of $M \cong \mathrm{SO}(n-1)$ on $\R^{n-1}$ is 
the natural action. In particular, it preserves the norm $\|\cdot\|$. 
For a fixed cusp $\mathfrak{a} = \gamma \xi \infty$, the subgroup of $\Gamma$,
denoted by $\Gamma_{\mathfrak{a}}$ is equal to $\gamma \Gamma_{\xi} \gamma^{-1}$, where 
$\Gamma_{\xi} =\Gamma \cap \xi N \xi^{-1}$, which is an unipotent subgroup. So we can replace 
$\gamma\xi \infty$ by $\gamma \xi u(\mathbf{n}) \infty$ where 
$u(\mathbf{n}) \in N \cap \xi^{-1}\Gamma \xi$. Then if $\|\y\| \gg e^r$ or $\|\y\| \ll e^r$, 
we may choose $\mathbf{n}$ appropriately such that $\|\y + e^r m \mathbf{n}\| \asymp e^r$,
then 
$$\begin{array}{rcl}
\gamma \xi u(\mathbf{n}) & = & u(\x) \sigma u(\y) a(r) m u(\mathbf{n}) \\
                         & = & u(\x)\sigma u(\y + e^r m \mathbf{n}) a(r) m .
     \end{array}
$$
So we can assume that $\|\y\| \asymp e^r$.
\end{remark}
\par Since $\mathbb{H}^n = G/K$, by Iwasawa decomposition, $ G/K = NAK/K \cong NA$. We will use the coordinate system in
~\cite{EKP}, denoting $u(\mathbf{x}) a(t)K/K$ by $(e^t, \x)$. We call it the $NA$-coordinate 
system.
\par In ~\cite{EKP}, the action of Weyl group element $\sigma$ on $G/K$ is given as follows:
\begin{equation}
 \label{weyl_group_action}
 \sigma (e^t, \mathbf{x})= \sigma u(\mathbf{x}) a(t)K/K = \frac{1}{e^{2t} + \|\mathbf{x}\|^2} (e^t, -\mathbf{x})
\end{equation}
\par We consider a typical geodesic ray $\mathcal{G}_{\mathbf{x}}= 
\{\mathcal{G}_{\mathbf{x}}(t): t > 0\}$(where 
$\mathbf{x} \in \R^{n-1}$) as follows:
$$\mathcal{G}_{\mathbf{x}}(t) = u(\mathbf{x})\sigma a(t).$$
\begin{definition}
\label{def_near_cusp_and_spec}
Take $s_1 > s_0 >0$ as in Theorem \ref{G_R}, we say $\mathcal{G}_{\mathbf{x}}$ enters (or is near)
the cusp $\mathfrak{a}= \gamma\xi \infty$ at $t$ if 
$$u(\mathbf{x})\sigma a(t) \in \gamma \xi N A_{s_1} K.$$
We say $\mathcal{G}_{\mathbf{x}}$ enters (or is near) the cusp $\mathfrak{a}= \gamma\xi \infty$,
if it enters (or is near) $\mathfrak{a}$ at $t$ for some $t >0$.
\par For $\x \in \R^{n-1}$, we define the spectrum of $\x$, denoted by $\mathrm{Spec}(\x)$ to be the 
sequence $\{\mathfrak{a}_i: i \in \mathbb{N}\}$ of cusp points $\mathcal{G}_{\x}$ enters, ordered by 
the time $t_i$ at which $\mathcal{G}_{\x}$ enters $\mathfrak{a}_i$. 
\end{definition}
\par Concerning the height of cusp point and geodesic rays near the cusp, we have the following 
proposition similar to rational convergents of real numbers:
\begin{proposition}
\label{condition_enter_cusp}
 If $\mathcal{G}_{\mathbf{x}}$ enters a cusp $\mathfrak{a}= \gamma \xi \infty$, then 
 $$\|\mathbf{x} - \mathfrak{a}\| \ll \frac{1}{h(\mathfrak{a})}.$$
 Here $\|\cdot\|$ denotes the Euclidean norm on $\R^{n-1}$. Conversely, there exists a
 constant $c >0$ such that if 
 $$\|\mathbf{x}- \mathfrak{a}\| \leq \frac{c}{h(\mathfrak{a})},$$
 then $\mathcal{G}_{\mathbf{x}}$ enters $\mathfrak{a}$.
\end{proposition}

\begin{remark} In the case $n=2$ and $\Gamma = \mathrm{SL}(2,\Z)$, the cusp points are 
rationals, and the height of a reduced rational $\frac{p}{q}$ ($(p,q)=1$) is $q^2$, and 
for $x \in \R$, 
a geodesic ray $\mathcal{G}_{x}$ enters $\frac{p}{q}$ if and only if 
$\frac{p}{q}$ is a convergent of $x$, which is equivalent to the above inequality
holds. Thus the above proposition extends this approximations of rationals
to any dimension $n \geq 2$ 
and any lattice $\Gamma < \mathrm{SO}(n,1)$.
\end{remark}
\begin{proof}
 Suppose at $t_0$, $\mathcal{G}_{\mathbf{x}}$ enters $\mathfrak{a} = \gamma\xi \infty$, then
 $u(\mathbf{x})\sigma a(t_0) =\gamma\xi n a(s)k$, where $e^{s} \asymp e^{s_0} \asymp 1$. Write
 $\gamma \xi =u(\mathbf{x}_1) \sigma u(\mathbf{x}_2) a(r)m$ in the Bruhat decomposition.
 Then 
 $$u(\mathbf{x})\sigma a(t_0) = u(\mathbf{x}_1)\sigma u(\mathbf{x}_2)a(r+s) k',$$
 where $k'=km \in K$. Rewrite it as follows:
 $$a(-r)u(-\mathbf{x}_2)\sigma u(\mathbf{x} - \mathbf{x}_1)\sigma a(t_0) = a(s) k'$$
 Put both sides into the $NA$-coordinate system, then
 $$\mathrm{LHS} = 
 \left(\frac{e^{-t_0 - r}}{e^{-2t_0} +\|\mathbf{x} - \mathbf{x}_1\|^2}, 
 e^{-r}\left( \frac{\mathbf{x}_1 - \mathbf{x}}{e^{-2t_0} +\|\mathbf{x}- \mathbf{x}_1\|^2} - \mathbf{x}_2\right)\right),$$
 and
 $$\mathrm{RHS} = (e^s, \mathbf{0}).$$
 Then $1 \asymp e^s = \frac{e^{-t_0 - r}}{e^{-2t_0} +\|\mathbf{x} - \mathbf{x}_1\|^2}$.
 We denote 
 $$
 \begin{array}{rcl }
 f(t_0) & =  & \frac{e^{-t_0 - r}}{e^{-2t_0} +\|\mathbf{x} - \mathbf{x}_1\|^2} \\
        & =  & \frac{e^{-r}}{e^{-t_0} + e^{t_0}\|\mathbf{x}- \mathbf{x}_1\|^2},
 \end{array}
$$
$f(t_0)$ is increasing when $t_0 \leq \|\mathbf{x} - \mathbf{x}_1\|^{-1}$ and decreasing 
when $t_0 > \|\mathbf{x} - \mathbf{x}_1\|^{-1}$. The maximum of $f(t_0)$ is equal to
$\frac{e^{-r}}{2\|\mathbf{x} - \mathbf{x}_1\|}$, so $1 \asymp f(t_0)$ implies that 
$$\frac{e^{-r}}{2\|\mathbf{x} - \mathbf{x}_1\|} \gg 1.$$
Note that $h(\mathfrak{a}) = e^r$ and $\mathfrak{a} = \mathbf{x}_1$ (because 
$\mathfrak{a} = \gamma \xi \infty \in u(\mathbf{x}_1) \sigma P $), then the above
inequality completes the first part of the proof.
\par Conversely, if the maximum of $f(t_0)$, say $\frac{e^r}{2\|\mathbf{x}- \mathbf{x}_1\|}$ 
is large enough, we could make $e^s \geq e^{s_1}$ for some $t_0$, in particular, 
$\mathcal{G}_{\mathbf{x}}(t_0)$ is near $\mathfrak{a}$. This proves the 
second part.
\end{proof}
\begin{remark} 
From the proof above, we note that the function $f(t_0)$ describes how deep a 
geodesic ray $\mathcal{G}_{\mathbf{x}}(t_0)$ enters into a cusp $\mathfrak{a}$:
when $f(t_0)$ is large the geodesic ray is near the cusp.
\end{remark}
\begin{definition}
\label{def_depth_function}
 Given a geodesic ray $\mathcal{G}_{\mathbf{x}}$ and a cusp point $\mathfrak{a}$, 
 we define the depth function $f_{\mathfrak{a}}(t, \mathbf{x})$ with 
 respect to $\mathfrak{a}$ as follows:
 $$f_{\mathfrak{a}}(t, \mathbf{x}) = \frac{1}{h(\mathfrak{a})(e^{-t} + e^{t}\|\mathbf{x}- \mathfrak{a}\|^2)}.$$
We define the total depth function $W(t,\mathbf{x})$ as follows:
$$W(t,\mathbf{x}) = \max_{\mathfrak{a} \in \Gamma\Xi\infty} \{f_{\mathfrak{a}} (t,\mathbf{x})\}.$$
\end{definition}
\begin{corollary}
 \label{time_enter_leave_cusp}
 Given $\x \in \R^{n-1}$ and a cusp point $\mathfrak{a}$ such that
 $\|\x - \mathfrak{a}\| h(\mathfrak{a})$ is small enough, then the time
 $t_1$ when $\mathcal{G}_{\x}$ enters $\mathfrak{a}$ satisfies:
 $$e^{t_1} \asymp h(\mathfrak{a}),$$
 the time $\tau$ when $f_{\mathfrak{a}}(t)$ has its maximum satisfies:
 $$e^{\tau} = \frac{1}{\|\x - \mathfrak{a}\|},$$
 and the time $t_2$ when $\mathcal{G}_{\x}$ leaves $\mathfrak{a}$ satisfies:
 $$e^{t_2} \asymp \frac{1}{h(\mathfrak{a}) \|\x - \mathfrak{a}\|^2}.$$
\end{corollary}

\begin{proof}
 We consider the function $f(t) = f_{\mathfrak{a}}(t, \x)$. It admits its unique maximum 
 $\frac{1}{h(\mathfrak{a}) \|\x - \mathfrak{a}\|}$ at 
 $t = - \log (\|\x - \mathfrak{a}\|) := \tau$. This proves the second equation.
\par When $t< \tau$, since $e^{-t} > e^t \|\x - \mathfrak{a}\|^2$, 
$$f(t) \asymp \frac{1}{h(\mathfrak{a}) e^{-t_0}}.$$ 
Suppose at $t = t_1$, $\mathcal{G}_{\x}$ enters $\mathfrak{a}$, then $t_1 < \tau$, and 
 $f(t_1) \asymp 1$, this shows that 
 $$e^{t_1} \asymp h(\mathfrak{a}).$$
 \par When $t > \tau$, since $e^{-t} < e^t \|\x - \mathfrak{a}\|^2$, 
 $$f(t)\asymp \frac{1}{h(\mathfrak{a})\|\x - \mathfrak{a}\|^2 e^t}.$$
 Suppose at $t = t_2$, $\mathcal{G}_{\x}$ leaves $\mathfrak{a}$, then $t_2 > \tau$, and 
 $f(t) \asymp 1$, this shows that 
 $$e^{t_2} \asymp \frac{1}{h(\mathfrak{a}) \|\x - \mathfrak{a}\|^2}.$$
 \par This completes the proof.

\end{proof}
\begin{remark} For $\x \in \R^{n-1}$, we have defined its spectrum 
$$\mathrm{Spec}(\x) = \{\mathfrak{a}(i, \x) : i \in \mathbb{N}\},$$
we denote by $t_1(i,\x)$, $\tau(i,\x)$ and $t_2(i,\x)$ the times when 
$\mathcal{G}_{\x}$ enters $\mathfrak{a}(i,\x)$, when 
$f_{\mathfrak{a}(i,\x)}(t,\x)$ admits its maximum, and when $\mathcal{G}_{\x}$
leaves $\mathfrak{a}(i,\x)$, respectively. Then according the above proposition,
$$e^{t_1(i,\x)} \asymp h(\mathfrak{a}(i,\x)),$$
$$e^{\tau(i,\x)} = \frac{1}{\|\x - \mathfrak{a}(i,\x)\|} ,$$
and 
$$e^{t_2(i,\x)} \asymp \frac{1}{h(\mathfrak{a}(i,\x))\|\x - \mathfrak{a}(i,\x)\|^2}.$$
\end{remark}
\par When a geodesic ray $\mathcal{G}_{\x}$ is in the compact part of some fundamental domain, we can still 
associate it to a cusp point of this fundamental domain, under this condition, we say the geodesic ray is 
roughly near the cusp point:
\begin{definition}
 \label{rough_enter_cusp_rough_spec}
 We can enlarge the cusp parts of $\Gamma \setminus \mathbb{H}^n$ such that their union
covers the whole space. Then in the space $\mathbb{H}^n$, we can divide each fundamental 
domain into several enlarged cusp parts such that they cover the whole $\mathbb{H}^n$.
 For $\x \in \R^{n-1}$, we say that the geodesic ray $\mathcal{G}_{\x}$ roughly enters a cusp
 $\mathfrak{a}$ (or is roughly near the cusp $\mathfrak{a}$) at time $t$ if 
 $\mathcal{G}_{\x}(t)$ is inside the enlarged cusp part associated with $\mathfrak{a}$. In
 this sense, we can define the rough spectrum of $\x$ as the sequence of cusps that 
 $\mathcal{G}_{\x}$ roughly enters.
\end{definition}
\begin{remark} Using the same argument as the proof of Proposition \ref{condition_enter_cusp}, it
is easily shown that:
\end{remark}
\begin{corollary}
\label{condition_roughly_enter_cusp}
if $\mathcal{G}_{\x}$ roughly enters a cusp point $\mathfrak{a}$, then
$$\|\x - \mathfrak{a}\| \ll \frac{1}{h(\mathfrak{a})}.$$
\end{corollary}
Moreover, following the argument in the proof of Corollary \ref{time_enter_leave_cusp}, we
can deduce the same result concerning the time $\mathcal{G}_{\x}$ roughly enters and leaves a 
cusp point:
\begin{corollary}
 \label{time_roughly_enter_leave_cusp}
 For $\x \in \R^{n-1}$, suppose 
 the associated geodesic ray $\mathcal{G}_{\x}$ roughly enters a cusp point $\mathfrak{a}$, then 
 the time $t_1$ when $\mathcal{G}_{\x}$ roughly enters $\mathfrak{a}$ and the time
 $t_2$ when $\mathcal{G}_{\x}$ roughly leaves $\mathfrak{a}$ satisfy
 $$e^{t_1} \asymp h(\mathfrak{a}),$$
 and 
 $$e^{t_2} \asymp \frac{1}{h(\mathfrak{a})\|\x - \mathfrak{a}\|^2},$$
 respectively.
\end{corollary}

\par Now we will define an alternative metric $d_{o}(\cdot, \cdot)$ (with respect to a 
base point $o \in \mathcal{F}_0$) and 
the height of a cusp point $\tilde{h}(\mathfrak{a})$, and prove that in a fixed ball $B \subset \R^{n-1}$,
$h(\mathfrak{a}) \asymp \tilde{h}(\mathfrak{a})$ and 
$\|\mathbf{x} - \mathbf{y}\| \asymp d_o (\mathbf{x}, \mathbf{y})$. These definitions naturally
come from the 
geometric structure of hyperbolic spaces and can be generalized to any Riemannian manifold with 
negative curvature. 
\par We firstly recall the basic theory of Busemann function.
\par For $\xi \in \partial \mathbb{H}^n$
and $x, y \in \mathbb{H}^n$, $B_{\xi} (x, y)$ is defined as follows:
$$B_{\xi} (x,y) =\lim_{t \rightarrow \infty} d_{\mathbb{H}} (x, \xi(t)) - d_{\mathbb{H}} (y, \xi(t)),$$ 
where $d_{\mathbb{H}}$ denotes the hyperbolic distance in $\mathbb{H}^n$ and $\{\xi(t)\}_{t \geq 0}$ is any geodesic ray pointing to $\xi$. 
It turns out that the value is independent of the choice of the ray $\{\xi(t)\}$.
\par For $B_{\xi} (x, y)$ we have the following basic properties:
\begin{enumerate}
 \item Let $\{\xi(t)\}_{t\geq 0}$ be a geodesic ray pointing to $\xi$, then we have $B_{\xi}(\xi(t_1), \xi(t_2)) = t_2 -t_1$ for any nonnegative numbers $t_1, t_2$.
 \item If $y = n_{\xi} x$, for some $n_{\xi} \in \mathcal{U}(\xi) =\{g \in G: g \text{ is unipotent and }g \xi =\xi\}$, then $B_{\xi} (x, y) =0$.
 \item $B_{\xi} (x, y) + B_{\xi} (y,z) = B_{\xi} (x,z)$, $B_{\xi} (x,y) = - B_{\xi} (y,x)$.
 \item For any $g \in G$ we have $B_{g \xi}(g x, g y) = B_{\xi}(x,y)$.
\end{enumerate}
\begin{definition}
 \label{def_tilde_height}
 Fix $o = K \in \mathcal{F}_0$. For a cusp $\mathfrak{a} = \gamma \xi \infty$,
 $$\tilde{h}(\mathfrak{a}) =\exp( B_{\xi \infty} (\gamma^{-1} o , o) ).$$ 
\end{definition}

\begin{definition}
 \label{def_gromov_distance}
  For two points $\xi_1, \xi_2 \in \partial \mathbb{H}^n$, 
  and any point $x \in \mathbb{H}^n$, 
  we define the Gromov metric between $\xi_1$ and
  $\xi_2$ respect to $x$ (denoted by $d_{x}(\xi_1, \xi_2)$) as follows:
  $$d_{x} (\xi_1, \xi_2) = \exp \left(-\frac{1}{2}\left(\lim_{t\rightarrow \infty} B_{\xi_1} (x, \xi(t)) + B_{\xi_2} (x, \xi (t))\right)\right),$$
  where $\{\xi(t): t \geq 0\}$ denotes any geodesic ray pointing to any point $\xi$ at infinity. The value is independent of the choice of
  $\{\xi(t) : t \geq 0\}$.
\end{definition}
\begin{remark} An important property of the Gromov metric is that it is uniformly bounded, namely, there exists a constant $M >0$ such that 
$$d_o(\xi_1, \xi_2) \leq M,$$
for all $\xi_1 , \xi_2 \in \partial \mathbb{H}^n$.
\end{remark}
\begin{proposition}
\label{equivalent_two_def}
Given a fixed compact subset $B \subset \R^{n-1}$, for any 
$\mathbf{x}, \mathbf{y} \in B$,
$$d_o(\mathbf{x}, \mathbf{y}) \asymp \|\mathbf{x}-\mathbf{y}\|,$$
and for any cusp point $\mathfrak{a} \in B$,
$$\tilde{h}(\mathfrak{a}) \asymp h(\mathfrak{a}).$$
\end{proposition}
\begin{proof}
For $\mathbf{x}, \mathbf{y} \in B$, choose $k, k' \in K$ such that 
$k a(t) \rightarrow mathbf{x}$ and $k' a(t) \rightarrow \mathbf{y}$ as 
$t \rightarrow \infty$. We could choose $k = u(\mathbf{x})\sigma n_1 a(t_1)$
and $k' = u(\mathbf{y})\sigma n_2 a(t_2)$. Then from $\mathbf{x}, \mathbf{y} \in B$,
it is easily seen that $n_1, n_2$ and $a(t_1), a(t_2)$ are uniformly bounded, in particular,
$e^{t_1} \asymp e^{t_2} \asymp 1$. Choose the geodesic ray 
$\{u(\mathbf{x})\sigma n_1 a(t_0) a(t): t \geq 0\}$. It is easily seen that 
$B_{\mathbf{x}} ( o,k a(t)K ) = t$. To find $d_{o} (\mathbf{x}, \mathbf{y})$ 
we also need to find $B_{\mathbf{y}}(o, k a(t)K/K)$. To do this, we need to find 
$s \in \R$ such that
$k a(t)K/K = k' n a(s)K/K$ for some $n \in N$ (this $s$ is exactly 
$B_{\mathbf{y}} (o, k a(t)K/K)$).
$$\begin{array}{l}
u(\mathbf{x}) \sigma n_1 a(t_0) a(t) = u(\mathbf{y}) \sigma n_2 n a(t_2) a(s) k" \\
u(\mathbf{x} - \mathbf{y}) \sigma u(\mathbf{x}_1) a(t+t_0) = \sigma u(\mathbf{z}) a(s+ t_2) k",
\end{array}
$$ 

for some $k" \in K$. Here $u(\mathbf{x}_1) = n_1$ and $u(\mathbf{z}) = n_2 n$.
\par  Compare the $NA$-coordinates of the two sides of the above equality, we have 
$$\mathrm{LHS} = \left(\frac{e^{t+ t_0}}{(e^{t+t_0})^2 + \|\mathbf{x}_1\|^2}, 
(\mathbf{x} - \mathbf{y}) - \frac{\mathbf{x}_1}{(e^{t+t_0})^2 + \|\mathbf{x}_1\|^2}\right)$$
$$\mathrm{RHS} = \left(\frac{e^{s+t_2}}{(e^{s+t_2})^2 + \|\mathbf{z}\|^2} , 
- \frac{x}{(e^{s+t_2})^2 + \|\mathbf{z}\|^2}\right)$$
Since $\|\mathbf{x}_1\|$ is uniformly bounded, when $t$ is large enough, 
we have $\frac{e^{t+ t_0}}{(e^{t+t_0})^2 + \|\mathbf{x}_1\|^2}$ is very close to
$\frac{1}{e^{t+t_0}}$ and
$ \frac{\mathbf{x}_1}{(e^{t+t_0})^2 + \|\mathbf{x}_1\|^2})$ 
is very small, which means $(\mathbf{x} - \mathbf{y}) - 
\frac{\mathbf{x}_1}{(e^{t+t_0})^2 + \|\mathbf{x}_1\|^2})$ 
is very close to
$\mathbf{x} - \mathbf{y}$. Therefore we have 
$\frac{e^{s+t_2}}{(e^{s+t_2})^2 + \|\mathbf{z}\|^2} \asymp \frac{1}{e^{t+t_0}}$ and 
$\frac{\|\mathbf{z}\|}{(e^{s+t_2})^2 + \|\mathbf{z}\|^2} \asymp \|\mathbf{x} - \mathbf{y}\|$, 
taking
the square sum of the above two estimates we have 
$$\frac{1}{(e^{s+t_2})^2 + \|\mathbf{z}\|^2} 
\asymp (\frac{1}{e^{t+t_0}})^2 + \|\mathbf{x} - \mathbf{y}\|^2 \asymp 
\|\mathbf{x} - \mathbf{y}\|^2$$
given $t$ large enough.
Therefore we have 
$$e^{s+t} \asymp e^{s+t+t_0 +t_2} \asymp (e^{s+t_2})^2 + \|\mathbf{z}\|^2 
\asymp \frac{1}{\|\mathbf{x} - \mathbf{y}\|^2}$$
This gives that
$$\exp\left(-\frac{1}{2}\lim_{t\rightarrow \infty}\left( B_{\mathbf{x}} (o, k a(t)K/K) 
+ B_{\mathbf{y}} (o, k a(t)K/K)\right) \right) \asymp \|\mathbf{x} -\mathbf{y}\|.$$
This completes the first part of the proposition.
\par To compute $\tilde{h}(\gamma \xi \infty)$, we choose $k \in K$ such that $\gamma^{-1} k a(t) \rightarrow \xi \infty$ as 
$t \rightarrow \infty$, and $k' \in K$ such that $k' a(s) \rightarrow \xi \infty$ as $s \rightarrow \infty$. Obviously, $k'$ corresponding to the 
vector $v_{\xi} \in T_{o}^1 \mathbb{H}^n $ pointing to $\xi \infty$, this means that $k'$ can only be chosen within a fixed finite subset 
of $K$, depending only on the subgroup $\Gamma$. And then there exists some $t$ such that $\gamma^{-1} k a(t) = \xi n \xi^{-1} k'$, this $t$ will be 
$\log \tilde{h} (\gamma\xi \infty)$ from the properties of Busemann function (note that $ \gamma^{-1}k a(t) = g_t(\gamma^{-1} k)$).
Now, $\gamma^{-1} k a(t) \rightarrow \xi \infty$ means that $\gamma^{-1} k  \in \xi P$ which implies $k \in \gamma \xi P$. For the same reason,
$k' \in \xi P$ which means $\xi^{-1} k' \in P$.
\par Let $\gamma \xi = u(\mathbf{x}_1) \sigma u(\mathbf{x}_2)a(r) m$, then from above we have $k= u(\mathbf{x}_1) \sigma n_k a(t_k) m_k$ for some 
$n \in N$, $m_k \in M$ and $a(t_k) \in A$. From our assumption, $\mathbf{x}_1 \in B$, this shows that $n_k \in N$ and $a(t_k)$ are both bounded by some
fixed compact subsets (depending on $K$ and $B$) of $N$ and $A$ respectively, since $k \in K$ and $u(\mathbf{x}_1) \in u(B)$ are both bounded inside some
compact subset $K$ and $u(B)$ respectively. In particular, $t_k \in [-C, C]$ for some constant $C > 0$ depending on $K$ and $B$. Also, from
$\xi^{-1}k' \in P$ we have $\xi^{-1}k' = n' a(t_{k'}) m' $, it is clear that $\xi^{-1} k'$ is contained some fixed finite subset, which implies 
$t_{k'} \in [-C' , C']$ for some absolute constant depending only on $\Gamma$.
\par Then from $k a(t) = \gamma \xi n \xi^{-1} k'$, we have 
$$u(\mathbf{x}_1) \sigma n_k a(t_k) m_k a(t) = u(\mathbf{x}_1) \sigma u(\mathbf{x}_2) a(r) m n n' a(t_{k'}) m' $$
\par In the Bruhat decomposition $g=n_1 \sigma n_2 a m$, compare the $a$-component of the left and right side of the above equation, we can
get $t = r + r_{k'} -r_k$, given that $r_k$ and $r_{k'}$ are both bounded from above and below, we have that
$$e^t \asymp e^r.$$
Since $e^t = \tilde{h}(\gamma \xi \infty)$ and $e^r = h(\gamma \xi \infty)$, this proves the second part of the proposition.
\end{proof}
Concerning $\tilde{h}(\cdot)$ and $d_o(\cdot,\cdot)$, we have the following interesting equality:
\begin{proposition}
 \label{invariant_under_group_action_height_metric}
 For any $\gamma, \gamma_1, \gamma_2 \in \Gamma$ and $\xi_1, \xi_2 \in \Xi$,
 $$d_o^2 (\gamma_1 \xi_1 \infty, \gamma_2 \xi_2 \infty)\tilde{h}(\gamma_1 \xi_1 \infty) \tilde{h}(\gamma_2 \xi_2 \infty)= d_o^2 (\gamma\gamma_1 \xi_1 \infty, \gamma \gamma_2 \xi_2 \infty) \tilde{h}(\gamma \gamma_1 \xi_1 \infty) \tilde{h}(\gamma\gamma_2 \xi_2 \infty)$$
\end{proposition}
\begin{proof}
 Let $\{\xi_1(t): t \geq 0\}$ be a geodesic ray pointing to $\xi_1\infty$, then
 \begin{equation}
  \begin{array}{cl}
    & B_{\gamma_1 \xi_1 \infty}(o, \gamma_1 o) + B_{\gamma_2 \xi_2\infty} (o, \gamma_2 o) - B_{\gamma_1 \xi_1 \infty}(o, \gamma_1 \xi_1(t)) - B_{\gamma_2\xi_2\infty}(o, \gamma_1\xi_1(t))\\
   = & (B_{\gamma_1 \xi_1 \infty}(o, \gamma_1 o)-B_{\gamma_1 \xi_1 \infty}(o, \gamma_1 \xi_1(t))) + (B_{\gamma_2 \xi_2\infty} (o, \gamma_2 o)- B_{\gamma_2\xi_2\infty}(o, \gamma_1\xi_1(t)))\\
   = & B_{\gamma_1 \xi_1 \infty}(\gamma_1 \xi_1 (t), \gamma_1 o) +B_{\gamma_2 \xi_2 \infty}(\gamma_1 \xi_1 (t), \gamma_2 o) \\
    & \text{(from the basic properties of Busemann function)} \\
    = & B_{\gamma\gamma_1 \xi_1 \infty}(\gamma\gamma_1 \xi_1 (t),\gamma \gamma_1 o) +B_{\gamma\gamma_2 \xi_2 \infty}(\gamma\gamma_1 \xi_1 (t), \gamma\gamma_2 o)\\
     & \text{(also from basic properties)}\\
   = & (B_{\gamma\gamma_1 \xi_1 \infty}(o, \gamma\gamma_1 o)-B_{\gamma\gamma_1 \xi_1 \infty}(o,\gamma \gamma_1 \xi_1(t))) + (B_{\gamma\gamma_2 \xi_2\infty} (o, \gamma\gamma_2 o)- B_{\gamma\gamma_2\xi_2\infty}(o, \gamma\gamma_1\xi_1(t)))\\
   = & B_{\gamma\gamma_1 \xi_1 \infty}(o,\gamma \gamma_1 o) + B_{\gamma\gamma_2 \xi_2\infty} (o, \gamma\gamma_2 o) - B_{\gamma\gamma_1 \xi_1 \infty}(o, \gamma\gamma_1 \xi_1(t)) - B_{\gamma\gamma_2\xi_2\infty}(o, \gamma\gamma_1\xi_1(t))
  \end{array}
 \end{equation}
By taking the limits of the first and the last expressions in the equations above and applying exponential function, we prove the statement. 
\end{proof}

\begin{corollary}
 \label{distance_between_two_cusps}
 For any two cusps $\gamma_1 \xi_1 \infty$ and $\gamma_2\xi_2 \infty$ inside $B$, we have
 $$\|\gamma_1\xi_1 \infty - \gamma_2\xi_2 \infty\| \gg \frac{1}{\sqrt{h(\gamma_1\xi_1\infty)h(\gamma_2\xi_2\infty)}}.$$
\end{corollary}
\begin{proof}
We at first prove the above inequality for $\tilde{h}(\cdot)$ and $d_o(\cdot,\cdot)$.
\par 
We at first consider the case $\gamma_1 = \gamma$ and $\gamma_2 = \mathrm{id}$. Choose a particular geodesic ray pointing to $\gamma \xi_1 \infty$, say $\{\gamma \xi_1 a(t): t \geq 0\}$, suppose we have 
 $\gamma \xi_1 = \xi_2 n_1 a(-r) k$ in the representation $G = \xi_2 N A K$, by the Theorem \ref{G_R}, we have $r$ is bounded 
 from below by some absolute constant,
 then we have 
 $$
 \begin{array}{c l}
  & B_{\xi_2\infty} (o, \gamma \xi_1 a(t)) \\
  = & B_{\xi_2\infty} (o, \xi_2 n_1 a(-r) k a(t))  
 \end{array}
 $$
 Projecting $n_1 a(-r) k a(t)$ onto $G/K$ and consider the first component of the $NA$-coordinate (put $k = n_1\sigma  n_2 a m$ in 
 Bruhat decomposition and do the same calculation as we did in the proof of Proposition \ref{equivalent_two_def}), 
 we have that it is equal to $\frac{e^{t-r}}{e^{2t}+\|\mathbf{y}\|^2}$
 for some vector $\mathbf{y} \in \mathbb{R}^{n-1}$, it is clear that it is at most $\frac{e^{t-r}}{e^{2t}} = e^{-t-r}$. Then we have 
 $$\xi_2 n_1 a(-r) k a(t) = \xi_2 n' a(-r-t - \epsilon(t)) k$$ for some $\epsilon(t) \geq 0$. This means that 
 $B_{\xi_2\infty} (o, \gamma \xi_1 a(t)) = -r-t -\epsilon(t) + t_1$, where $t_1$ is such that $\xi_2 = k_2 a(t_2) n_2$.
 \par For $B_{\gamma \xi_1 \infty} (o, \gamma \xi_1 a(t))$, we have that
 $$
 \begin{array}{cl}
 & B_{\gamma \xi_1 \infty} (o, \gamma \xi_1 a(t)) \\
 = & B_{\gamma \xi_1 \infty} (o, \gamma \xi_1) + B_{\gamma \xi_1 \infty} (\gamma \xi_1, \gamma \xi_1 a(t)) \\
  = & B_{\gamma \xi_1 \infty}(o, \gamma o) +B_{\gamma \xi_1 \infty} (\gamma o, \gamma \xi_1) +t \\
  = & B_{\gamma \xi_1 \infty}(o, \gamma o) + B_{\xi_1 \infty}(o, \xi_1) +t 
 \end{array}
 $$
 We denote $B_{\xi_1 \infty} (o, \xi_1)$ by $C(\xi_1)$, since it is contained in a fixed finite set, it is uniformly bounded.
\par Thus, we have 
$$
\begin{array}{cl}
 & B_{\gamma \xi_1 \infty}( o, \gamma o) - B_{\xi_2 \infty}(o, \gamma \xi_1 a(t)) - B_{\gamma\xi_1 \infty}(o, \gamma \xi_1 a(t)) \\
 = & B_{\gamma \xi_1 \infty}( o, \gamma o) - B_{\gamma \xi_1 \infty}(o, \gamma o) -C(\xi_1) -t +t +r + \epsilon(t) -t_1 \\
 = & r+ \epsilon (t) -t_1 -C(\xi_1) \\
 \geq & \tilde{C}.
\end{array}
$$
for some absolute constant $\tilde{C}$. By letting $t \rightarrow \infty$ and taking the exponential in the above inequality, we have
$$d^2_o(\gamma \xi_1 \infty, \xi_2 \infty) \tilde{h}(\gamma  \xi_1 \infty ) \gg 1.$$
This proves the statement since $\tilde{h}(\xi_2 \infty) =1$.
\par 
Now we consider the general case. For any two cusps $\gamma_1 \xi_1 \infty$ and $\gamma_2 \xi_2 \infty$, we have 
$$
\begin{array}{cl}
 & \tilde{h}(\gamma_1 \xi_1 \infty) \tilde{h}(\gamma_2\xi_2 \infty) d_o^2(\gamma_1\xi_1\infty, \gamma_2\xi_2\infty) \\
 = & \tilde{h}(\gamma_2^{-1}\gamma_1 \xi_1 \infty) \tilde{h}(\xi_2\infty) d_o^2(\gamma_2^{-1}\gamma_1 \xi_1 , \xi_2 \infty) \\
 \gg & 1.
\end{array}
$$
\par 
Now applying the facts that $h(\gamma\xi\infty) \asymp \tilde{h} (\gamma\xi\infty)$ and $d_o(\gamma_1\xi_1\infty, \gamma_2\xi_2\infty) \asymp \|\gamma_1\xi_1\infty - \gamma_2 \xi_2 \infty\|$
whenever the cusps are in $B$, we prove the inequality for $h(\cdot)$ and $\|\cdot\|$.
\end{proof}
\begin{definition}
\label{def_successive_cusp}
For a cusp $\mathfrak{a} = \gamma \xi \infty$, we call another cusp $\mathfrak{b}$ is a successive cusp of $\mathfrak{a}$ if
$$\mathfrak{b} = \gamma (\xi u(\mathbf{n}) \xi^{-1}) s \xi' \infty,$$ where $\mathbf{n}\in \R^{n-1}$,
$\xi u(\mathbf{n}) \xi^{-1} \in \Gamma \cap \xi N \xi^{-1} := \Gamma_{\xi}$ (it is a cocompact lattice of $\xi N \xi^{-1}$) 
with the norm $\|\mathbf{n}\|$ large enough, and $s \not\in \Gamma_{\xi}$ is chosen from a finite subset $\mathcal{S} \subset \Gamma$ defined as follows: for any fundamental domain $\mathcal{F}$ 
that shares a common boundary with $\mathcal{F}_0$ is of form $s \mathcal{F}_0$, there are only finitely many such fundamental domains, 
we define $\mathcal{S}$ to be the collection of all the possible $s$'s. It is easily seen that $\Gamma$ is generated by $\mathcal{S}$.
\end{definition}

\begin{proposition}
\label{height_distance_of_successive_cusp}
Given a cusp $\mathfrak{a} = \gamma \xi_1 \infty$ and one of its 
successive cusp points $\mathfrak{b} = \gamma \xi_1 u(\mathbf{n}) \xi_1^{-1} s \xi_2 \infty$, the following approximation is true:
$$h(\mathfrak{b}) \asymp h(\mathfrak{a})\|\mathbf{n}\|^2.$$
And moreover, 
$$\|\mathfrak{a} - \mathfrak{b}\| \asymp \frac{1}{\sqrt{h(\mathfrak{a}) h(\mathfrak{b})}}$$
\end{proposition}
\begin{proof}
 Let $\gamma \xi_1 = u(\mathbf{x}_1)\sigma u(\mathbf{y}_1) a(r_1) m_1$ and
$\gamma \xi_1 u(\mathbf{n}) \xi_1^{-1} s \xi_2= u(\mathbf{x}_2) \sigma u(\mathbf{y}_2) a(r_2) m_2$, we take the inverse of both sides, then the left hand side becomes 
$$(\xi_1^{-1} s \xi_2)^{-1} u(-\mathbf{n}) (\gamma \xi_1)^{-1} = k^{-1}(\xi_1, \xi_2, s) u(-\mathbf{n}) m_1^{-1} a(-r_1) u(-\mathbf{y}_1) \sigma u(-\mathbf{x}_1)$$
where $k(\xi_1, \xi_2, s)$ denotes $\xi_1^{-1} s \xi_2$ for $\xi_1, \xi_2 \in \Xi$ and $s\in \mathcal{S}$. we claim that $k(\xi_1, \xi_2 ,s) \not\in P$ unless $s =e$ and
$\xi_1 = \xi_2$. This is because if this happens then $\xi_1 \infty = s \xi_2 \infty$, this means that they represent exact the same cusp.
\par We denote $k(\xi_1, \xi_2, s) = u(\mathbf{x}(\xi_1, \xi_2, s)) \sigma a(r(\xi_1, \xi_2, s)) u(\mathbf{y}(\xi_1,\xi_2, s)) m(\xi_1,\xi_2,s)$, since we only have finitely many choices for
$\xi_1$, $\xi_2$ and $s$, we have $\mathbf{x}(\xi_1, \xi_2,s)$, $r(\xi_1, \xi_2, s)$ and $\mathbf{y}(\xi_1, \xi_2, s)$ are all bounded. This makes the left hand side equal
$$m^{-1}(\xi_1, \xi_2,s)u(-\mathbf{y}(\xi_1, \xi_2, s)) a(-r(\xi_1 , \xi_2, s)) \sigma u(-\mathbf{x}(\xi_1, \xi_2, s)) u(-\mathbf{n}) m_1^{-1} a(-r_1) u(-\mathbf{y}_1) \sigma u(-\mathbf{x}_1)$$
The right hand side is equal to
$$m_2^{-1} a(-r_2) u(-\mathbf{y}_2) \sigma u(-\mathbf{x}_2)$$
\par Now we consider their $NA$-coordinates on $G/K$, the right hand side has coordinate
$$\left(\frac{e^{-r_2}}{1+\|\mathbf{x}_2\|^2}, e^{-r_2}m_2^{-1}\left(\frac{\mathbf{x}_2}{1+\|\mathbf{x}_2\|^2} -\mathbf{y}_2\right)\right).$$
The left hand side is equal to:
$$\left(e^{-r(\xi_1, \xi_2, s)} \frac{A}{ A^2 + \|\mathbf{B}\|^2},m^{-1}(\xi_1, \xi_2, s)
\left( -\mathbf{y}(\xi_1, \xi_2, s) + e^{-r(\xi_1, \xi_2, s)}\frac{\mathbf{B}}{A^2 + \|\mathbf{B}\|^2}\right)\right),$$
where $A = \frac{e^{-r_1}}{1+\|\mathbf{x}_1\|^2}$ and $\mathbf{B} = m_1^{-1}\left( \frac{\mathbf{x}_1}{1+\|\mathbf{x}_1\|^2} - \mathbf{y}_1\right) - \mathbf{n} -\mathbf{x}(\xi_1, \xi_2, s)$,
since we assume that $\|\mathbf{n}\|$ is large enough, we have $\|\mathbf{B}\| \asymp \|\mathbf{n}\|$, and since $\|\mathbf{x}_1\|\ll 1$ and 
$e^{r(\xi_1, \xi_2, s)} \asymp 1$, we have that
$$e^{-r(\xi_1, \xi_2, s)} \frac{A}{A^2 +\|\mathbf{B}\|^2} \asymp \frac{e^{-r_1}}{\|\mathbf{n}\|^2}$$
and the first coordinate of right hand side $\frac{e^{-r_2}}{1+\|\mathbf{x}_2\|^2} \asymp e^{-r_2}$ since $\|\mathbf{x}_1\| \ll 1$. Therefore, we have 
$$e^{r_2} \asymp e^{r_1}\|\mathbf{n}\|^2.$$
\par This proves the first part of the proposition.
\par Moreover, by comparing the second component of the coordinate, we have 
$$e^{-r_2}m_2^{-1}\left(\frac{\mathbf{x}_2}{1+\|\mathbf{x}_2\|^2} -\mathbf{y}_2\right) = m^{-1}(\xi_1, \xi_2, s)
\left( -\mathbf{y}(\xi_1, \xi_2, s) + e^{-r(\xi_1, \xi_2, s)}\frac{\mathbf{B}}{A^2 + \|\mathbf{B}\|^2}\right).$$
The right hand side has uniformly bounded norm, this shows that the left hand side is also uniformly bounded. Thus, 
$$\|\frac{\mathbf{x}_2}{1+\|\mathbf{x}_2\|^2} -\mathbf{y}_2\| \ll e^{r_2}.$$
It is easily seen that the norm of $\frac{\mathbf{x}_2}{1+\|\mathbf{x}_2\|^2}$ is uniformly bounded. 
This shows that 
$$\|\mathbf{y}_2\| \ll e^{r_2} = h(\mathfrak{b}).$$
Remark: here we could not assume that $\|\mathbf{y}_2\| \ll h(\mathfrak{b})$ (as we mentioned in the remark after Definition \ref{def_height}), 
since
this will change $\gamma \xi_1 u(\mathbf{n}) \xi_1^{-1} s \xi_2 $ to $\gamma \xi_1 u(\mathbf{n}) \xi_1^{-1} s \xi_2 u(\mathbf{n}')$ for some 
$\mathbf{n}' \in \R^{n-1}$, then the following equality will not hold anymore:
$$u(\mathbf{x}_1)\sigma u(\mathbf{y}_1) a(r_1) m_1 u(\mathbf{n}) \xi_1^{-1} s \xi_2= u(\mathbf{x}_2) \sigma u(\mathbf{y}_2) a(r_2) m_2.)$$

\par Now we start with 
$$\gamma \xi_1 u(\mathbf{n}) \xi_1^{-1} s \xi_2 = u(\mathbf{x}_1)\sigma u(\mathbf{y}_1) a(r_1) m_1 u(\mathbf{n})k(\xi_1, \xi_2, s),$$
on the other hand,
$$\gamma \xi_1 u(\mathbf{n}) \xi_1^{-1} s \xi_2 = u(\mathbf{x}_2) \sigma u(\mathbf{y}_2) a(r_2) m_2.$$
Suppose the $NA$-coordinate of $k(\xi_1, \xi_2, s)$ is $(e^r(\xi,s), \mathbf{z}(\xi,s))$, then the $NA$-coordinate of 
$\gamma \xi_1 u(\mathbf{n}) \xi_1^{-1} s \xi_2$ is the following, by plugging in the first equation:
$$\left( \frac{C}{C^2 + \|\mathbf{D}\|^2} , \mathbf{x}_1 - \frac{\mathbf{D}}{C^2 +\|\mathbf{D}\|^2}\right),$$
where $C = e^{r_1 + r(\xi, s)} \asymp h(\mathbf{a})$, $\mathbf{D} = \mathbf{y}_1 + e^{r_1} m_1 (\mathbf{n} + \mathbf{z}(\xi, s))$,
for $\|\mathbf{n}\|$ large enough, $\|\mathbf{D}\| \asymp h(\mathfrak{a}) \|\mathbf{n}\|$.
So 
$$\left\|\frac{\mathbf{D}}{C^2 +\|\mathbf{D}\|^2}\right\| 
\asymp \frac{1}{h(\mathfrak{a})\|\mathbf{n}\|}.$$
\par By plugging the second equation into the $NA$-coordinate, we have 
$$\left(\frac{e^{r_2}}{e^{2r_2} + \|\y_2\|^2} , \x_2 - \frac{\y_2}{e^{2r_2} + \|\y_2\|^2} \right).$$
We have proved that $\|\y_2\| \ll e^{r_2} = h(\mathfrak{b})$, so 
$$\left\|\frac{\y_2}{e^{2r_2} + \|\y_2\|^2} \right\| \ll \frac{1}{h(\mathfrak{b})} \ll \frac{1}{h(\mathfrak{a})\|\mathbf{n}\|}.$$
By comparing the second component of their $NA$-coordinates, we get 
$$\mathbf{x}_1 - \frac{\mathbf{D}}{C^2 +\|\mathbf{D}\|^2} = \x_2 - \frac{\y_2}{e^{2r_2} + \|\y_2\|^2} ,$$
therefore
$$\x_1 - \x_2 = \frac{\mathbf{D}}{C^2 +\|\mathbf{D}\|^2} - \frac{\y_2}{e^{2r_2} + \|\y_2\|^2},$$
our above argument shows that the right hand side has norm $\asymp \frac{1}{h(\mathfrak{a})\|\mathbf{n}\|} \asymp 
\frac{1}{\sqrt{h(\mathfrak{a}) h(\mathfrak{b})}}$, this shows that 
$$\|\x_1 - \x_2\| \asymp \frac{1}{\sqrt{h(\mathfrak{a}) h(\mathfrak{b})}} .$$
This proves the second part of the proposition.
\end{proof}

\begin{corollary}
 \label{leaves_cusp_enter_successive_cusp}
 There exist constants $\epsilon >0$ and $C >0$ such that for a cusp $\mathfrak{a}$ and one of its 
 successive cusp points $\mathfrak{b}$, then for any 
 $$\x \in B\left(\mathfrak{b}, \frac{\epsilon}{h(\mathfrak{b})}\right) =\left\{\x \in \R^{n-1}: 
 \|\x - \mathfrak{b}\| < \frac{\epsilon}{h(\mathfrak{b})}\right\} ,$$
 $\mathcal{G}_{\x}$ enters both $\mathfrak{a}$ and $\mathfrak{b}$, moreover, from 
 leaving $\mathfrak{a}$ to entering $\mathfrak{b}$, it spends at most time $C$.
\end{corollary}
\begin{proof}
 For $\epsilon >0$ small enough, then if 
 $$\|\x - \mathfrak{b}\| \leq \frac{\epsilon}{h(\mathfrak{b})},$$
 the following is also true:
 $$\|\x - \mathfrak{a}\| \leq \frac{c}{h(\mathfrak{a})},$$
 since $\|\mathfrak{b} -\mathfrak{a}\| \asymp \frac{1}{\sqrt{h(\mathfrak{a}) h(\mathfrak{b})}} 
 \ll \frac{1}{h(\mathfrak{a})} $
 (see Proposition \ref{height_distance_of_successive_cusp}). Therefore 
 $\mathcal{G}_{\x}$ enters both $\mathfrak{a}$ and $\mathfrak{b}$.
 \par The time $t_2(\mathfrak{a})$ when $\mathcal{G}_{\x}$ leaves $\mathfrak{a}$
 satisfies:
 $$e^{t_2(\mathfrak{a})} \asymp \frac{1}{h(\mathfrak{a})\|\x - \mathfrak{a}\|^2} \asymp h(\mathfrak{b}),$$
 and the time $t_1(\mathfrak{b})$ when $\mathcal{G}_{\x}$ enters $\mathfrak{b}$ 
 satisfies:
 $$e^{t_1(\mathfrak{b})} \asymp h(\mathfrak{b}).$$
 Therefore $e^{t_2(\mathfrak{a})} \asymp e^{t_1(\mathfrak{b})}$, which is equivalent to our
 conclusion.
\end{proof}
\par For $\x \in \R^{n-1}$, if we consider the rough spectrum of $\x$, say $\{\mathfrak{a}_i: i \in \mathbb{N}\}$, we have the 
following result:
\begin{proposition}
\label{distance_between_point_and_rough_cusp}
 Let $\x \in \R^{n-1}$, we denote the rough spectrum of $\x$ by $\{\mathfrak{a}_i: i \in \mathbb{N}\}$, then for each $i \in \mathbb{N}$,
 $$\|\x - \mathfrak{a}_i\| \asymp \frac{1}{\sqrt{h(\mathfrak{a}_i) h(\mathfrak{a}_{i+1})}}.$$
 And moreover, for each $i \in \mathbb{N}$, $\mathfrak{a}_{i+1}$ is a successive cusp of $\mathfrak{a}_i$
\end{proposition}
\begin{proof}
 By Corollary \ref{time_roughly_enter_leave_cusp}, the time $t$ when $\mathcal{G}_{\x}$ roughly
 leaves $\mathfrak{a}_i$ and roughly enters $\mathfrak{a}_{i+1}$ satisfies 
 $$e^t \asymp \frac{1}{h(\mathfrak{a}_i)\|\x - \mathfrak{a}_i\|^2}$$
 and 
 $$e^t \asymp h(\mathfrak{a}_{i+1}),$$
 this implies that 
 $$\|\x - \mathfrak{a}_{i+1}\| \asymp \frac{1}{\sqrt{h(\mathfrak{a}_i) h(\mathfrak{a}_{i+1})}}.$$
 The proof of the second statement goes as follows: suppose when $\mathcal{G}_{\x}$ roughly enters $\mathfrak{a}_i$, it is in the fundamental 
 domain $\gamma \mathcal{F}_0$, then $\mathfrak{a}_i = \gamma \xi \infty$, where $\xi \infty$ is some cusp point of $\mathcal{F}_0$. When 
 $\mathcal{G}_{\x}$ leaves $\mathfrak{a}_i$, the fundamental domain it leaves
 must contain $\gamma \xi \infty$ as a cusp point, therefore the fundemantal domain is of form $\gamma \xi u(\mathbf{n})\xi^{-1} \mathcal{F}_0$
 where $\xi u(\mathbf{n}) \xi^{-1} \in \Gamma_{\xi}$. Then when $\mathcal{G}_{\x}$ roughly enters $\mathfrak{a}_{i+1}$, the fundamental domain 
 it enters must be adjacent to $\gamma \xi u(\mathbf{n})\xi^{-1}$, this implies the cusp point $\mathfrak{a}_{i+1}$ must be of form 
 $\gamma \xi u(\mathbf{n})\xi^{-1} s \xi'\infty$ for some $s \in \mathcal{S}$ and some other cusp $\xi' \infty$ of $\mathcal{F}_0$. This completes 
 the proof.
\end{proof}

\section{Counting cusp points in a given region}
In this section we will prove the counting result on cusp points inside a 
given region with heights in a given range, as we mentioned in the introduction (see Problem \ref{counting problem}).
\begin{theorem}
 \label{counting_cusp_thm}
  There exist constants $A_3>0$, $A_1 < 1 < A_2$, $T>0$, $h > 0$ and $\Upsilon >0$ such that for any $t \geq T$, 
  and for any $\gamma \in \Gamma$ satisfying
the cusp $\mathfrak{a} =\gamma\infty \in B$, and $h(\mathfrak{a}) \geq h$, 
we have the number of cusp points $\mathfrak{b}$ of form $\gamma'\infty$ such that
 $h(\mathfrak{b}) \in [A_1 e^t h(\mathfrak{a}), A_2 e^t h(\mathfrak{a})]$ 
 and $\mathfrak{b} \in B\left(\gamma\infty, \frac{A_3}{h(\mathfrak{a})}\right)$ is at least
 $\Upsilon e^{(n-1)t}$. 
\end{theorem}
\begin{remark} The basic argument of the proof is based on the idea in the thesis of Margulis on 
counting closed geodesics in compact Riemannian manifold of negative curvature. The basic
tool is the mixing of geodesic flow on hyperbolic space of finite volume.
\end{remark}
\par The basic idea of the proof goes as follows: given $\epsilon >0$, we take a neighborhood $\Omega$ of 
$\mathrm{id}$ in $G$ of form $\mathcal{N}_{\epsilon} \mathcal{A}_{\epsilon} \mathcal{U}^{-}_{\epsilon} M$, where 
$\mathcal{N}_{\epsilon}$, $\mathcal{A}_{\epsilon}$ and $\mathcal{U}^{-}_{\epsilon}$ are $\epsilon$-neighborhoods of $\mathrm{id}$
in $N$, $A$ and $U^{-}$, respectively, such that it maps to $\Gamma \setminus G$ injectively under the natural projection:
$$\pi: G \rightarrow \Gamma \setminus G.$$
From the mixing property of geodesic flow on $\mathrm{T}^1(\Gamma \setminus \mathbb{H}^n)$ with respect to the Lebesgue measure $\mu_G$, 
we have when $t >0$ large enough, 
$$\mu_G (\Gamma \Omega a(t)  \cap \Gamma \Omega ) \geq \frac{9}{10} (\mu_{G}(\Gamma \Omega ))^2.$$
Unfolding the above intersection to $\mathrm{T}^1(\mathbb{H}^n)$, the left hand side is equal to
$$\sum_{\gamma' \in \Gamma} \mu(\gamma \Omega a(t)  \cap \gamma' \Omega ),$$
where $\mu$ denotes the $G$-invariant Lebesgue measure on $\mathrm{T}^1(\mathbb{H}^n)$. 
By a result proved in ~\cite{Gorodnik_Shah}, 
$$\mu(\gamma \Omega a(t)  \cap \gamma' \Omega ) \leq C(\Omega) e^{-(n-1)t},$$
where $C(\Omega) >0$ is a constant depending on $\Omega$. Then there exists a constant $\Upsilon >0$
such that there are at least $\Upsilon e^{(n-1)t}$
$\gamma' \in \Gamma$ such that 
$$\gamma \Omega a(t) \cap \gamma' \Omega \neq \emptyset. $$
Each such $\gamma'$ will be proved to satisfy the properties described in Theorem \ref{counting_cusp_thm}, 
which finishes the proof.
\par We start with proving the following lemma:
\begin{lemma}
 \label{lemma_cusp_in_given_range}
There exist constants $\epsilon >0$, $A_3 > 0$, $A_2 > 1 > A_3$, such that:
 let $\Omega \subset G$ denote the $\epsilon$-neighborhood of $\mathrm{id}$ of form
$\mathcal{N}_{\epsilon} \mathcal{A}_{\epsilon} \mathcal{U}^{-}_{\epsilon} M$,
given any $\gamma \in \Gamma$ such that $\mathfrak{a} = \gamma \infty \in B$ and 
 $h(\mathfrak{a})$ is large enough, $t >0$ large enough, any $\gamma'$ such that 
 $$\gamma \Omega a(t) \cap \gamma' \Omega \neq \emptyset$$
 satisfies the following:
 \begin{itemize}
  \item $\|\gamma' \infty - \mathfrak{a} \| \leq \frac{A_3}{h(\mathfrak{a})}$.
  \item $h(\gamma'\infty) \in [A_1 e^t h(\mathfrak{a}), A_2 e^t h(\mathfrak{a})]$.
 \end{itemize}

\end{lemma}

\begin{proof}

 By the definition of $\gamma'$, there exist $w_1, w_2 \in \Omega$ such that 
 $$\gamma w_1 a(t) = \gamma' w_2 .$$ 
\par We use $\tilde{h}(\gamma \infty)$ and $\tilde{h}(\gamma'\infty)$. Suppose $\gamma = k a(r) u(\mathbf{x})$, then from the definition of 
$\tilde{h}(\cdot)$, $\tilde{h}(\gamma \infty) = e^r$. Since $\Gamma_{\infty} =\Gamma \cap N$ is a lattice of $N$, we could replace $\gamma$ by $\gamma u(\mathbf{n})$
for appropriate $u(\mathbf{n}) \in \Gamma_{\infty}$ to make $\|\mathbf{x}\| \asymp 1$.
\par Suppose $w_1 = u(\x_1) a(\epsilon_1) u^{-}(\y_1) m_1$, $w_2 =u(\x_2) a(\epsilon_2) u^{-}(\y_2) m_2$, where 
$\|\epsilon_i\| < \epsilon$, $\|\x_i\|\leq \epsilon$ and $\|\y_i\| \leq \epsilon$ for $i=1,2$. Then 
$$\gamma u(\x_1) a(\epsilon_1) u^{-}(\y_1)m_1 a(t)m^{-1}_2 u^{-}(-\y_2) = \gamma' u(\x_2) a(\epsilon_2). $$
Plugging in $\gamma = k a(r) u(\x)$, the left hand side is equal to 
$$\mathrm{LHS} = k' a(r+\epsilon_1) u(e^{-\epsilon_1} m (\x + \x_1)) u^{-}(m \y_1 - e^{-t} \y_2 ) a(t),$$
where $m= m_2 m_1^{-1}$, and $k' = km^{-1} \in K$. Denote $\tilde{r} = r+\epsilon_1$, $e^{-\epsilon_1} m(\x + \x_1) = \tilde{\x}$,
and denote $\tilde{\y} = m \y_1 - e^{-t} \y_2$, then $\tilde{r}$ is close to $r$, $\|\tilde{\x}\| \asymp 1$ and $\|\tilde{\y}\|$ is very small,
then 
$$\mathrm{LHS} = k' a(\tilde{r}) u(\tilde{\x}) u^{-}(\tilde{\y}) a(t),$$
we want to write $k' a(\tilde{r}) u(\tilde{\x}) u^{-}(\tilde{\y})$ 
in terms of $KAN$-decomposition. To do this, we consider its inverse
$$u^{-}(-\tilde{\y}) u(-\tilde{\x}) a(-\tilde{r}) k'^{-1} $$
and calculate its $NA$-coordinate: we 
at first write $u^{-}(-\tilde{\y}) = \sigma u(-\tilde{\y}) \sigma$,
then direct computation shows that its $NA$-coordinate is the following:
$$\sigma \left( \frac{e^{-\tilde{r}}}{e^{-2\tilde{r}} +\|\tilde{\x}\|^2}, 
\frac{\tilde{\x}}{e^{-2\tilde{r}} +\|\tilde{\x}\|^2} - \tilde{\y}\right),$$
we denote 
$$\eta = \frac{e^{-\tilde{r}}}{e^{-2\tilde{r}} +\|\tilde{\x}\|^2} ,$$
and 
$$\mathbf{Z} = \frac{\tilde{\x}}{e^{-2\tilde{r}} +\|\tilde{\x}\|^2} - \tilde{\y},$$
it is easily seen that $\eta \asymp e^{-\tilde{r}}$, and $\|\mathbf{Z}\|$ is close 
to $\|\tilde{x}\|^{-1} \asymp 1$. Then after applying the action of $\sigma$ on 
$(\eta, \mathbf{Z})$, we get 
$$\left(\frac{\eta}{\eta^2 + \|\mathbf{Z}\|^2}, - \frac{\mathbf{Z}}{ \eta^2 + \|\mathbf{Z}\|^2}\right),$$
the $A$-component is 
$$\frac{\eta}{\eta^2 + \|\mathbf{Z}\|^2} \asymp \frac{\eta}{\|\mathbf{Z}\|^2} \asymp e^{-r},$$
and the $N$-component is 
$$- \frac{\mathbf{Z}}{ \eta^2 + \|\mathbf{Z}\|^2} := - \tilde{\mathbf{Z}}$$
has norm $\|\tilde{\mathbf{Z}}\| \asymp \frac{1}{\|\mathbf{Z}\|} \asymp 1$.
This shows that 
$$k' a(\tilde{r}) u(\tilde{\x}) u^{-}(\tilde{\y}) = k'' a(r') u(\tilde{\mathbf{Z}}),$$
where $|r - r'|$ is bounded by a constant. Then 
$$k' a(\tilde{r}) u(\tilde{\x}) u^{-}(\tilde{\y}) a(t) = k'' a(r'+ t) u(e^{-t} \tilde{\mathbf{Z}}),$$
therefore
$$\gamma' u(\x_2) a(\epsilon_2) = k'' a(r'+ t) u(e^{-t} \tilde{\mathbf{Z}}).$$
This easily implies 
$$\gamma' = k'' a(r'+ t -\epsilon_2) u(e^{\epsilon_2 -t}\tilde{\mathbf{Z}} - \x_2),$$
this shows that $\tilde{h}(\gamma'\infty) = e^{r' + t - \epsilon_2} \asymp e^t \tilde{h}(\gamma \infty)$.
\par Next we want to prove that 
$$\|\gamma' \infty - \gamma \infty\| \leq \frac{A_3}{h(\mathfrak{a})}.$$
\par To do this, we firstly consider the Gromov metric $d_o(\cdot, \cdot)$, and then
make use of the fact that
$$d_o(\xi_1, \xi_2) \asymp \|\xi_1 - \xi_2\|$$ 
when $\xi_1, \xi_2 \in B$. 
\par 
For a noncusp point $\gamma w_1 \infty$ with $w_1 \in \Omega$, we could define
the $\tilde{h}(\cdot)$ height of $\gamma w_1 \infty$ similarly 
$$\tilde{h}(\gamma w_1 \infty) = \exp (B_{w_1 \infty} (\gamma^{-1} o, o)).$$
Then the above argument shows that 
$$\tilde{h}(\gamma w_1 \infty) \asymp \tilde{h}(\gamma \infty).$$
Then Proposition \ref{invariant_under_group_action_height_metric} tells that
$$
\begin{array}{cl} & d_o^2(\gamma \infty, \gamma w_1 \infty) \tilde{h}(\gamma \infty) 
\tilde{h}(\gamma w_1 \infty) \\ =  & 
d^2_o(\infty, w_1 \infty) \tilde{h}(\infty) \tilde{h}(w_1 \infty) \\
= & d^2_o(\infty, w_1 \infty) \\
\ll & 1
\end{array}$$
The last inequality above follows from the basic properties of Gromov metric. So we have 
$$d_o(\gamma \infty, \gamma w_1 \infty) \ll 
\frac{1}{\tilde{h}(\gamma \infty) } .$$
Similarly, we have 
$$d_o(\gamma' \infty , \gamma' w_2 \infty) \ll \frac{1}{\tilde{h}(\gamma'\infty)}.$$
Since $\gamma' w_2 \infty = \gamma w_1 a(t) \infty = \gamma w_1 \infty$, and since 
$\tilde{h}(\gamma'\infty) \asymp e^t \tilde{h}(\gamma \infty)$, the following holds:
$$d_o(\gamma \infty, \gamma' \infty) \ll \frac{1}{\tilde{h}(\gamma \infty)}.$$
This completes the proof because when $\gamma \infty, \gamma'\infty \in B$, 
$$\|\gamma\infty - \gamma'\infty\|\asymp d_o(\gamma\infty, \gamma'\infty),$$ 
$$\tilde{h}(\gamma\infty) \asymp h(\gamma \infty),$$ 
and
$$\tilde{h}(\gamma'\infty) \asymp h(\gamma' \infty).$$
\end{proof}

\par We will need the following result of Gorodnik and Shah:

\begin{proposition}(See ~\cite[Proposition 3.2]{Gorodnik_Shah})
\label{prop_gorod_shah}
Let $G$ be a real algebraic group, and $\sigma$ is an involution of $G$. Let $A=\{a(t)\}$ be a one parameter subgroup of $G$, such that $\sigma(a(t))= a(-t)$. Let 
\begin{equation}
\begin{array}{ll}
H=\{ g \in G: \sigma(g) =g\} & U^{+}= \{g \in G : a(-t)g a(t) \rightarrow e \text{ as } t \rightarrow \infty\} \\
U^{-}= \{g \in G : a(t)g a(-t) \rightarrow e \text{ as } t \rightarrow \infty\} & Z=Z_{G}(A)
\end{array}
\end{equation}
Then there exist constant $c$ and $t_0$ such that for neighborhoods $H_{r_1} \subset H$, $Z_{r_2} \subset Z$ and $U^{+}_{r_3}$ of $e$ in $H$, $Z$ and $U^{+}$, respectively, small enough, $t> t_0$ and any $g \in G$, we have
\begin{equation}
\mu_G(g U^{+}_{r_3} Z_{r_2}H_{r_1}\cap U^{+}_{r_3} Z_{r_2}H_{r_1}a(-t)) \leq c e^{-\lambda t}\mu_{U^{+}}(U^{+}_{r_3})^2
\end{equation}
where $\mu_{U^{+}}$ denotes the Haar measure on $U^{+}$, and $\lambda$ is the sum of the eigenvalues of $\mathrm{Ad}(a(1))$ which are
greater than $1$.
\end{proposition}
\begin{remark} In our case, we take $G = \mathrm{SO}(n,1)$ and $\sigma$ be the 
weyl element, then we have $H = K \cong \mathrm{SO}(n)$ and $\lambda = n-1$. We could choose $r_1$, $r_2$ and 
$r_3$ appropriately such that $\Omega \subset U^{+}_{r_3} Z_{r_2}H_{r_1}$. Then the above proposition tells that 
$$\mu(\gamma \Omega a(t)  \cap \gamma' \Omega ) \leq C(\Omega) e^{-(n-1)t},\quad (\ast)$$
for some constant $C(\Omega)$ depending on $\Omega$.
\end{remark}
\begin{proof}[proof of Theorem \ref{counting_cusp_thm}]
As we mentioned before, the diagonal flow $A = \{a(t): t \in \R\}$ is mixing on $\Gamma \setminus G$ with respect to the finite $G$-invariant
measure $\mu_{G}$. So there exists a constant $T >0$ such that for $t > T$, 
$$\mu_G(\Gamma\Omega a(t) \cap \Gamma \Omega) \geq \frac{9}{10} (\mu_G(\Gamma \Omega))^2.$$
Let $\gamma \in \Gamma$ be as above, then by unfolding $\Gamma \Omega a(t) \cap \Gamma \Omega$ to $G$, we have
$$\mu_G(\Gamma \Omega a(t) \cap \Gamma \Omega) = \sum_{\gamma' \in \Gamma} \mu(\gamma\Omega a(t)\cap \gamma' \Omega). $$
By $(\ast)$, we get 
$|\{\gamma' \in \Gamma: \gamma \Omega a(t) \cap \gamma' \Omega \neq \emptyset\}| \geq \Upsilon e^{(n-1)t}$
for some constant $\Upsilon >0$. Here $|\cdot|$ denotes the cardinality of a set. Then applying Lemma \ref{lemma_cusp_in_given_range}, we finish the 
proof of Theorem \ref{counting_cusp_thm}, because any such $\gamma'$ satisfies:
\begin{itemize}
  \item $\|\gamma' \infty - \mathfrak{a} \| \leq \frac{A_3}{h(\mathfrak{a})}$,
  \item $h(\gamma'\infty) \in [A_1 e^t h(\mathfrak{a}), A_2 e^t h(\mathfrak{a})]$,
 \end{itemize}
 for some constants $0 < A_1 < 1 < A_2$ and $A_3 >0$.
\end{proof}

\section{Hausdorff dimension of Divergent trajectories under diagonal geodesic flow}

In this section, we will compute the Hausdorff dimension of $\mathfrak{D}_k$. In the first 
subsection, we will give the lower bound of the Hausdorff dimension, and the second subsection
will be devoted to the proof of upper bound of the Hausdorff dimension.
\par Given $\mathcal{V}_k = (\x_1, \x_2, \dots, \x_k) \in \R^{n-1}$, we denote by $\mathcal{G}(\mathcal{V}_k)$ the diagonal geodesic ray in 
$\mathcal{M}_k$:
$$\mathcal{G}(\mathcal{V}_k) := \{\mathcal{G}(\mathcal{V}_k, t) = (\Gamma_1 u(\x_1)\sigma a(t), \dots, \Gamma_k u(\x_k)\sigma a(t)): t \geq 0\}.$$
We define the function $W(\mathcal{V}_k, t)$ as follows:
$$W(\mathcal{V}_k, t) = \max_{1\leq i\leq k} \{W_i(\x_i, t)\},$$
where $W_i(\x, t)$ denotes the function $W(\x, t)$ defined in Definition \ref{def_depth_function}, 
with respect to the space $\Gamma_i \setminus \mathbb{H}^n$. Then 
$\mathcal{G}(\mathcal{V}_k, t)$ diverges as $t \rightarrow \infty$ if and only if 
$W(\mathcal{V}_k,t) \rightarrow \infty$ as $t \rightarrow \infty$.
\par Because we are interested in divergent trajectories, we may assume
that $W(\mathcal{V}_k, t)$ remains large for all $t>0$ large enough.
\par Suppose at some point $t$, 
$$W(\mathcal{V}_k, t) = W_i(\x_i, t),$$
then there exists a maximal interval $I$ such that for all $s \in I$,
$$W(\mathcal{V}_k,s) = W_i(\x_i,s).$$
Then from the properties of the function $W_i(\x_i, t)$, 
$$W_i(\x_i, t) = \frac{1}{h(\mathfrak{a})(e^{-t} + e^t \|\x_i - \mathfrak{a}\|^2)},$$
where $\mathfrak{a}$ is the cusp point $\mathcal{G}_{\x_i}(t)$ is close to. Then 
$t = - \log \|\x_i - \mathfrak{a}\| \in I$. 
\par Suppose $s \in I$ is the right limit of $I$, then at $s$, $W(\mathcal{V}_k, t)$ changes 
from $W_i(\x_i, t)$ to $W_j(\x_j,t)$, then 
$$W_i(\x_i, t) = \frac{1}{h(\mathfrak{a}) (e^{-t} + e^t \|\x_i - \mathfrak{a}\|^2)},$$
is equal to 
$$W_j(\x_j,t) = \frac{1}{h(\mathfrak{b})(e^{-t} + e^t \|\x_j - \mathfrak{b}\|^2)}.$$
Obviously $-\log\|\x_i - \mathfrak{a}\| < t < - \log\|\x_j -\mathfrak{b}\|$, then
$$W_i(\x_i , t) \asymp \frac{1}{h(\mathfrak{a}) e^t \|\x_i -\mathfrak{a}\|},$$
and 
$$W_j(\x_j, t) \asymp \frac{1}{h(\mathfrak{b}) e^{-t}},$$
this implies that
$$h(\mathfrak{a}) e^t \|\x_i - \mathfrak{a}\|^2 \asymp e^{-t} h(\mathfrak{b}) \asymp \sqrt{h(\mathfrak{a})h(\mathfrak{b})}\|\x_i -\mathfrak{a}\|.$$
Then $W(\mathcal{V}_k, t)$ has a local minimum $\frac{1}{\sqrt{h(\mathfrak{b})h(\mathfrak{a})} \|\x_i -\mathfrak{a}\|}$ when 
$e^t \asymp \frac{\sqrt{h(\mathfrak{b})}}{\sqrt{h(\mathfrak{a})} \|\x_i -\mathfrak{a}\|} $. $W(\mathcal{V}_k,t) \rightarrow \infty$ if
and only if these local minima $\frac{1}{\sqrt{h(\mathfrak{b})h(\mathfrak{a})} \|\x_i -\mathfrak{a}\|}$ tends to $\infty$ as $t \rightarrow \infty$.
\subsection{Lower bound of the Hausdorff dimension}
The basic idea to get the lower bound of the Hausdorff dimension is the following: it suffices to consider the case $k=2$, because if the 
projection of $\mathcal{G}(\mathcal{V}_k)$ to the first two component is divergent, then so is $\mathcal{G}(\mathcal{V}_k)$.
For the first component, we choose $\x_1$ such that 
$\mathrm{Spec}(\x_1)$ admits a subsequence $\{\mathfrak{a}_i: i \in \mathbb{N}\}$ such that
\begin{itemize}
 \item $h(\mathfrak{a}_{i+1}) \asymp h^{1+\delta}(\mathfrak{a}_i)$ for all $i \in \mathbb{N}$, where $\delta>0$ is some constant.
 \item From $\mathcal{G}_{\x_1}$ leaving $\mathfrak{a}_i$ to it entering $\mathfrak{a}_{i+1}$, it spends at most a uniformly bounded time $C>0$.
\end{itemize}
For such $\x_1$, 
$$\|\x_1 - \mathfrak{a}_i\| \asymp \frac{1}{\sqrt{h(\mathfrak{a}_i) h(\mathfrak{a}_{i+1})}} \asymp \frac{1}{h^{1+\delta/2}(\mathfrak{a}_i)},$$
and we could divide $\R$ into disjoint union of a sequence of intervals $\R=\bigcup_{i=1}^{\infty} I_i$, each $I_i = [t_i, t_{i+1})$ where 
$$e^{t_i} \asymp \frac{1}{\|\x_1 - \mathfrak{a}_i\|}$$
for each $t_i$, such that for $t \in I_i$
$$W_1(\x_1, t) \asymp \frac{1}{h(\mathfrak{a}_i)(e^{-t} +e^t\|\x_1 - \mathfrak{a}_i\|^2 )}.$$
\par For each such $\x_1$ fixed, we choose $\x_2$ on the second component inductively as follows:
\begin{itemize}
 \item We start with a cusp $\mathfrak{b}_{p}$ for some large $p$, such that $h(\mathfrak{b}_p) \asymp \frac{h(\mathfrak{a}_p)}{\log h(\mathfrak{a}_p)}$, and 
 take a neighborhood $B\left(\mathfrak{b}_p, \frac{1}{h(\mathfrak{a}_p)}\right)$. We denote 
 $A_{p} = B\left(\mathfrak{b}_p, \frac{1}{h(\mathfrak{a}_p)}\right)$, and $\mathfrak{A}_p = \{A_p\}$.
 \item Suppose we have defined a collection of finitely many neighborhoods $\mathfrak{A}_k$, and each neighborhood of 
 $\mathfrak{A}_k$ is of form
 $A_k =B\left(\mathfrak{b}_k, \frac{1}{h(\mathfrak{a}_k)}\right)$, where $\mathfrak{b}_k$ 
 is some cusp point of $\Gamma_2\setminus \mathbb{H}^n$ with
 $h(\mathfrak{b}_k) \asymp \frac{h(\mathfrak{a}_k)}{\log h(\mathfrak{a}_k)}$. For each $\mathfrak{b}_{k+1} \in A_k$ with 
 $h(\mathfrak{b}_{k+1}) \asymp \frac{h(\mathfrak{a}_{k+1})}{\log h(\mathfrak{a}_{k+1})}$, we construct the neighborhood 
 $B\left(\mathfrak{b}_{k+1}, \frac{1}{h(\mathfrak{a}_{k+1})}\right)$, and define $\mathfrak{A}_{k+1}$ 
 to be the collection of all such neighborhoods.
 This defines $\mathfrak{A}_n$ inductively for all $n\in \mathbb{N}$.
\end{itemize}
Denote $\mathcal{A}_k = \bigcup_{A_k \in \mathfrak{A}_k} A_k$, and define $A_{\infty} = \bigcap_{k} \mathcal{A}_k$. 
We will prove the following lemma:
\begin{lemma}
 \label{lemma_second_component_give_div_traj}
 Fix $\x_1$ and define $A_{\infty}$ as above,  then for any $\x_2 \in A_{\infty}$, 
 $\mathcal{G}(\x_1, \x_2)$ is a divergent trajectory.
\end{lemma}

\begin{proof}
 For any $\x_2 \in A_{\infty}$, there exists a sequence of cusp points $\{\mathfrak{b}_k: k \geq p\}$ such that 
 \begin{itemize}
  \item $h(\mathfrak{b}_k) \asymp \frac{h(\mathfrak{a}_k)}{\log h(\mathfrak{a}_k)}$, for all $k$.
  \item $\x_2 \in B\left(\mathfrak{b}_k, \frac{1}{h(\mathfrak{a}_k)}\right)$.
 \end{itemize}
Then from $e^t \asymp h(\mathfrak{a}_k) \log h(\mathfrak{a}_k)$ to 
$e^t = \frac{1}{\|\x_1 - \mathfrak{a}_k\|}$,
$$
\begin{array}{cl}
 & W_1(\x_1, t) \\
 = & \frac{1}{h(\mathfrak{a}_k) (e^{-t} + e^t \|\x_1 - \mathfrak{a}_k\|^2)} \\
 \asymp & \frac{1}{h(\mathfrak{a}_k) e^{-t}} \\
 \gg & \frac{h(\mathfrak{a}_k) \log h(\mathfrak{a}_k)}{h(\mathfrak{a}_k)} \\
 = & \log h(\mathfrak{a}_k).
\end{array}
$$
From $e^t = \frac{1}{\|\x_1 -\mathfrak{a}_k\|}$ to $e^t \asymp \frac{h(\mathfrak{a}_{k+1})}{\log h(\mathfrak{a}_{k+1})}$,
$$
\begin{array}{cl}
 & W_1(\x_1, t) \\
 \asymp & \frac{1}{e^t h(\mathfrak{a}_k)\|\x_1 - \mathfrak{a}_k\|^2} \\
 \gg & \frac{\log h(\mathfrak{a}_{k+1})}{h(\mathfrak{a}_{k+1})h(\mathfrak{a}_k)\|\x_1 - \mathfrak{a}_k\|^2} \\
 \asymp & \log h(\mathfrak{a}_{k+1}).
\end{array}
$$
From $e^t \asymp \frac{h(\mathfrak{a}_{k+1})}{\log h(\mathfrak{a}_{k+1})} \asymp h(\mathfrak{b}_{k+1})$ 
to $e^t = \frac{1}{\|\x_2 - \mathfrak{b}_{k+1}\|} \asymp h(\mathfrak{a}_{k+1})$, 
$$W_2(\x_2, t) \asymp \frac{e^t}{h(\mathfrak{b}_{k+1})} \asymp \frac{e^t \log h(\mathfrak{a}_{k+1})}{ h(\mathfrak{a}_{k+1})},$$
and it is increasing. We claim that at some $\tau_{k+1}$ such that 
$$e^{2\tau_{k+1}} \asymp \frac{h(\mathfrak{b}_{k+1})}{ h(\mathfrak{a}_k) \|\x_1 - \mathfrak{a}_k\|^2} \asymp 
h(\mathfrak{b}_{k+1}) h(\mathfrak{a}_{k+1}),$$
$W(\mathcal{V}_2, t)$ admits a local minimum and 
$$W(\mathcal{V}_2, \tau_{k+1}) = W_2(\x_2, \tau_{k+1}) \asymp \sqrt{\frac{h(\mathfrak{a}_{k+1})}{h(\mathfrak{b}_{k+1})}}  
\asymp \sqrt{\log h(\mathfrak{a}_{k+1})}.$$
Since $\sqrt{\log h(\mathfrak{a}_{k+1})} \rightarrow \infty$ as $t \rightarrow \infty$, we have 
the diagonal geodesic ray $\mathcal{G}(\x_1, \x_2)$ is divergent.
\end{proof}
Therefore, once we choose $\x_1$ as above and then choose a $\x_2\in A_{\infty}$ for this fixed $\x_1$, we will get a divergent 
trajectory. We will then compute the Hausdorff dimension of the collection of such $\x_1$'s, and for each $\x_1$ fixed, we will give 
the Hausdorff dimension of $A_{\infty}$, by the following lemma, the Hausdorff dimension of divergent trajectories is at least the 
sum of them:
\begin{lemma}[Marstrand Slicing Theorem] 
\label{slicing_theorem}
Let $A$ and
$B$ be metric spaces, and let $C$ be a subset of the direct product $A \times B$. Assume that
the projection of $C$ onto $A$, $\mathrm{Proj}_A (C)$ has Hausdorff dimension at least $\alpha >0$,
and for every $a \in \mathrm{Proj}_A (C)$, if we define
$$B_{a} = C \cap (\{a\}\times B)$$
then the Hausdorff dimension $\dim_H(B_{a})\geq \beta >0$ for all $a \in \mathrm{Proj}_A (C)$,
then we have
$$\dim_H C \geq \alpha +\beta.$$
\end{lemma}
\begin{remark} The reader may see ~\cite[Section 1.4]{Kleinbock_Margulis}, ~\cite{Mars}, and
~\cite[Theorem 5.8]{Falc} for the detail of this theorem.
\end{remark}
\par We at first prove that for each fixed $\x_1$ as above, the set $A_{\infty}$ defined as above has Hausdorff dimension
$n-1$:
\begin{theorem}
 \label{Hausdorff_dimension_second_component}
Given constants $\delta >0$ and $C>0$, let $\x_1 \in \R^{n-1}$ satisfy the following condition: its spectrum $\mathrm{Spec}(\x_1)$ admits a sebsequence 
 $\{\mathfrak{a}_k: k \in \mathbb{N}\}$, such that
 \begin{itemize}
  \item $h(\mathfrak{a}_{k+1}) \asymp h^{1+\delta}(\mathfrak{a}_k)$ for all $k \in \mathbb{N}$.
  \item From $\mathcal{G}_{\x_1}$ leaving $\mathfrak{a}_k$ to it entering $\mathfrak{a}_{k+1}$, it spends at most time 
  $C$.
 \end{itemize}
According to this $\x_1$, we construct $\mathfrak{A}_k$ inductively as above, and then define 
$$\mathcal{A}_k = \bigcup_{A \in \mathfrak{A}_k} A,$$
and
$$A_{\infty} = \bigcap_{k} \mathcal{A}_k,$$
then the Hausdorff dimension of $A_{\infty}$
$$\dim_H A_{\infty} \geq n-1 .$$
\end{theorem}
\par We introduce the notion of Cantor-like collection of compact subsets of $\mathbb{R}^{n-1}$ as follows:
\begin{definition}
\label{def_cantor_like_subset}
Starting with a bounded closed subset $A_0$ with positive Lebesgue measure, 
a Cantor-like countable collection $\mathfrak{A}$ is the union of finite collections
$\mathfrak{A}_k$ of compact subsets of $A_0$, for $k\in \mathbb{N}$, satisfying the following conditions:
\begin{enumerate}
 \item $\mathfrak{A}_0 =\{A_0\}$
 \item every $\mathfrak{A}_k$ is a finite collection of disjoint compact subsets of $A_0$.
 \item for every $k \geq 1$, for every $A \in \mathfrak{A}_k$, we can find some $B \in \mathfrak{A}_{k-1}$ such that $A \subset B$.
 \item let $d_k(\mathfrak{A}) = \sup_{A \in \mathfrak{A}_k} \mathrm{diam} (A)$, where $\mathrm{diam} (A)$ denotes the diameter of $A$, then 
 $d_k(\mathfrak{A}) \rightarrow 0$ as $k \rightarrow \infty$.
\end{enumerate}
For a Cantor-like collection $\mathfrak{A} = \bigcup_{k =0}^{\infty} \mathfrak{A}_k$, let $\mathcal{A}_k = \bigcup_{A \in \mathfrak{A}_k} A$, and 
$A_{\infty} = \bigcap_{k=0}^{\infty} \mathcal{A}_k$, and we define $\Delta_k (\mathfrak{A})$ as folows:
$$\Delta_k (\mathfrak{A}) = \inf_{B \in \mathfrak{A}_k} \frac{m(B \cap \mathcal{A}_{k+1})}{m(B)}$$
where $m(\cdot)$ denotes the Lebesgue measure of $\mathbb{R}^{n-1}$.
\end{definition}
The basic tool of the proof is the following theorem:
\begin{theorem}{(See ~\cite[Section 4.1]{Kleinbock_Margulis})}
 \label{tool_for_lower_bound_hausdorff_dimension}
 Let $\mathfrak{A} = \bigcup_{k=0}^{\infty} \mathfrak{A}_k$ be a Cantor-like collection of compact subsets of $A_0$, and let $\mathcal{A}_k$, 
 $A_{\infty}$, $d_k(\mathfrak{A})$ and $\Delta_k(\mathfrak{A})$ be as above, then we have the Hausdorff dimension of $A_{\infty}$
 $$\dim_{H} (A_{\infty}) \geq n-1 - \limsup_{j\rightarrow \infty} \frac{\sum_{i=0}^{j-1} \log (\frac{1}{\Delta_i(\mathfrak{A})})}{\log (\frac{1}{d_j(\mathfrak{A})})}$$
\end{theorem}
\begin{remark} The statement given in ~\cite[Section 4.1]{Kleinbock_Margulis} is more general 
than the version above, and the result was proved in ~\cite{McMullen} and 
~\cite{Urbanski}.
\end{remark}
\par Now we are ready to prove Theorem \ref{Hausdorff_dimension_second_component}:
\begin{proof}[Proof of Theorem \ref{Hausdorff_dimension_second_component}]
 To apply Theorem \ref{tool_for_lower_bound_hausdorff_dimension}, we need to estimate $\Delta_k(\mathfrak{A})$.
 \par Given each $A_k \in \mathfrak{A}_k$, suppose $A_k =B\left(\mathfrak{b}_k, \frac{1}{h(\mathfrak{a}_k)}\right)$,
 where $h(\mathfrak{b}_k) \asymp \frac{h(\mathfrak{a}_k)}{\log h(\mathfrak{a}_k)}$. We need to count how many cusp points 
 $\mathfrak{b}_{k+1} \in A_k$ with $h(\mathfrak{b}_{k+1}) \asymp \frac{h(\mathfrak{a}_{k+1})}{\log h(\mathfrak{a}_{k+1})}$.
 \par To do this, we at first find a successive cusp point $\mathfrak{b}'_k$ of $\mathfrak{b}_k$ in 
 $A_k$ such that 
 $$h(\mathfrak{b}'_k) \asymp h(\mathfrak{a}_k) \log h(\mathfrak{a}_k),$$
 then 
 $$\|\mathfrak{b}_k - \mathfrak{b}'_k\| \asymp \frac{1}{h(\mathfrak{a}_k)}.$$
 Let constants $T$, $\Upsilon$, $A_1$, $A_2$ and $A_3$ be as in Theorem \ref{counting_cusp_thm}, and define 
 $A'_k = B\left(\mathfrak{b}'_k, \frac{A_3}{h(\mathfrak{b}'_k)}\right)$, then $A'_k \subset A_k$, and applying 
 Theorem \ref{counting_cusp_thm}, for $t >T$, there are at least $\Upsilon e^{(n-1)t}$ cusp points $\mathfrak{b}_{k+1}$'s in 
 $A'_k$ with $h(\mathfrak{b}_{k+1}) \in [A_1 e^t h(\mathfrak{b}'_k), A_2 e^t h(\mathfrak{b}'_k)]$. Let 
 $$e^t = \frac{h(\mathfrak{b}_{k+1})}{h(\mathfrak{b}'_k)} \asymp \frac{h^{\delta}(\mathfrak{a}_k)}{\log^2 h(\mathfrak{a}_k)},$$
we get the number of choices for 
$\mathfrak{b}_{k+1}$ is $\asymp \frac{h^{(n-1)\delta}(\mathfrak{a}_k)}{\log^{2(n-1)} h(\mathfrak{a}_k)}$. 
\par Therefore 
$$\frac{1}{\Delta_k(\mathfrak{A})} \asymp \frac{\log^{2(n-1)} h(\mathfrak{a}_k)}{h^{(n-1)\delta}(\mathfrak{a}_k)} h^{(n-1)\delta}(\mathfrak{a}_k) =\log^{2(n-1)} h(\mathfrak{a}_k),$$
so 
$$\log \left(\frac{1}{\Delta_k(\mathfrak{A})}\right) =O(k),$$
and because 
$$\log \left(\frac{1}{d_k(\mathfrak{A})}\right) = (1+\delta)^k \log h(\mathfrak{a}_0) + O(k),$$
we have 
$$\begin{array}{cl}
   & \dim_H A_{\infty} \\
   \geq & n-1 - \limsup_{j\rightarrow \infty} \frac{\sum_{i=0}^{j-1} O(i)}{(1+\delta)^j \log h(\mathfrak{a}_0) + O(j)} \\
   = & n-1 - \limsup_{j\rightarrow \infty} \frac{O(j^2)}{(1+\delta)^j \log h(\mathfrak{a}_0) + O(j)} \\
   = & n-1.
  \end{array}
$$
This completes the proof.
\end{proof}

\par We give the lower bound of the Hausdorff dimension of collection of eligible $\x_1$'s in the following proposition:
\begin{proposition}
 \label{Hausdorff_dimension_first_component}
 For every $\delta >0$, we define the $\mathfrak{D}_{\delta}$ to be the 
 collection of $\x_1 \in \R^{n-1}$ satisfying the following:
$$\dim_{H} \mathfrak{D}_{\delta} \geq \frac{n-1}{2+\delta}$$ 
\end{proposition}

\begin{proof}
We will at first construct the Cantor-like collection $\mathfrak{A} = \bigcup \mathfrak{A}_k$.
\par Let the constant $C>0$ be the same as those in Corollary \ref{leaves_cusp_enter_successive_cusp}.
Starting with a cusp point $\mathfrak{a}_0$ with height $h(\mathfrak{a}_0)$ large enough, we define $A_0$ to be the closed ball 
centered at $\mathfrak{a}_0$ with radius $\frac{1}{h^{1+\delta/2}(\mathfrak{a}_0)}$, i.e., 
$$A_0 = B\left(\mathfrak{a}_0, \frac{1}{h^{1+\delta/2}(\mathfrak{a}_0)}\right).$$
We define $\mathfrak{A}_0 = \{A_0\}$.
\par Suppose we have defined $\mathfrak{A}_k$ for $k \geq 0$, and every 
$A_k \in \mathfrak{A}_k$ is of form 
$B\left(\mathfrak{a}_k, \frac{1}{h^{1+\delta/2}(\mathfrak{a}_k)}\right)$. We fix one such $A_k$.
We take $\mathfrak{a}_{k+1}$ to be a successive cusp point of $\mathfrak{a}_k$ 
(see Definition \ref{def_successive_cusp}) such that 
$h(\mathfrak{a}_{k+1}) \asymp h^{1+\delta}(\mathfrak{a}_k)$, and define 
$A_{k+1} =B \left(\mathfrak{a}_{k+1}, \frac{1}{h^{1+\delta/2}(\mathfrak{a}_{k+1})}\right)$. 
From 
Proposition \ref{height_distance_of_successive_cusp}, we have that 
$$\|\mathfrak{a}_{k+1} - \mathfrak{a}_k\| 
\asymp \frac{1}{\sqrt{h(\mathfrak{a}_k) h(\mathfrak{a}_{k+1})}} \asymp 
\frac{1}{h^{1+\delta/2}(\mathfrak{a}_k)},$$
so $A_{k+1} \subset A_k$ if we choose the constants appropriately in the approximation
$$h(\mathfrak{a}_{k+1}) \asymp h^{1+\delta}(\mathfrak{a}_k).$$
We take all possible $A_k$'s in $\mathfrak{A}_k$ 
and construct all possible $A_{k+1}$'s as above, and define $\mathfrak{A}_{k+1}$ to be 
the collection of all such $A_{k+1}$'s.
\par This finishes the inductive construction of $\mathfrak{A} = \bigcup_{k = 0}^{\infty} \mathfrak{A}_k$, 
and thus
$\mathcal{A}_k = \bigcup_{A \in \mathfrak{A}_k} A$ and 
$A_{\infty} =\bigcap_{k=0}^{\infty} \mathcal{A}_k$ are defined accordingly.
\par We will prove that $A_{\infty} \subset \mathfrak{D}_{\delta}$. Take any $\x \in  A_{\infty}$, then there exists a sequence 
$\{A_k = B\left(\mathfrak{a}_k, \frac{1}{h^{1+\delta/2}(\mathfrak{a}_k)}\right) 
\in \mathfrak{A}_k: n \in \mathbb{N} \}$, such that $\x \in \bigcap_{k=0}^{\infty} A_k$.
By Corollary \ref{leaves_cusp_enter_successive_cusp}, the geodesic ray $\mathcal{G}_{\x}$ enters
$\mathfrak{a}_k$ consequently, and moreover, from leaving $\mathfrak{a}_k$ to entering 
$\mathfrak{a}_{k+1}$, it spends at most time $C$. Because 
$h(\mathfrak{a}_{k+1}) \asymp h^{1+\delta}(\mathfrak{a}_k)$, this shows that 
$\x \in \mathfrak{D}_{\delta}$.
\par Next we will apply Theorem \ref{tool_for_lower_bound_hausdorff_dimension} to give the lower bound
of the Hausdorff dimension of $A_{\infty}$.
\par Take any $A_k \in \mathfrak{A}_{k}$ of form 
$B\left(\mathfrak{a}_k, \frac{1}{h^{1+\delta/2}(\mathfrak{a}_k)}\right)$, suppose that 
$$\mathfrak{a}_k= \gamma \xi \infty ,$$
we want to count how many successive cusp points $\mathfrak{a}_{k+1}$'s we could choose. 
From Definition \ref{def_successive_cusp}, every $\mathfrak{a}_{k+1}$ is of form
$$\mathfrak{a}_{k+1} = \gamma \xi u(\mathbf{n}) \xi^{-1} s \xi' \infty, $$
and from Proposition \ref{height_distance_of_successive_cusp},
$h(\mathfrak{a}_{k+1}) \asymp h(\mathfrak{a}_k)\|\mathbf{n}\|^2$, therefore we have 
$\|\mathbf{n}\| \asymp h^{\delta/2}(\mathfrak{a}_k)$. So the number of choices 
for $\mathfrak{a}_{k+1}$ is $\asymp h^{(n-1)\delta/2}(\mathfrak{a}_k)$. 
Therefore 
$$\Delta_k(\mathfrak{A}) \asymp h^{(n-1)\delta/2}(\mathfrak{a}_k) 
\frac{h^{(n-1)(1+\delta/2)}(\mathfrak{a}_k)}{h^{(n-1)(1+\delta)(1+\delta/2)}(\mathfrak{a}_k)} = h^{-(n-1)(1+\delta)\delta/2}(\mathfrak{a}_k). $$
And 
$$d_k(\mathfrak{A}) \asymp \frac{1}{h^{1+\delta/2}(\mathfrak{a}_k)}.$$
So 
$$\log \left(\frac{1}{\Delta_k(\mathfrak{A})}\right) = (n-1)(1+\delta)\delta/2 
\log h(\mathfrak{a}_k) = (n-1)(1+\delta)^{k+1} \delta/2 \log (h(\mathfrak{a}_0)) + O(k),$$
 and 
$$\log \left(\frac{1}{d_k(\mathfrak{A})}\right) = 
(1+\delta/2) \log h(\mathfrak{a}_k) = (1+\delta/2)(1+\delta)^k 
\log (h(\mathfrak{a}_0)) + O(k)$$
where $O(k)$ denotes some quantity depending on $k$ such that $|O(k)| \ll k$ for large $k$.
\par Then from Theorem \ref{tool_for_lower_bound_hausdorff_dimension},
$$\begin{array}{cl}
   & \dim_H A_{\infty} \\
   \geq & n-1 - \limsup_{j\rightarrow \infty} \frac{\sum_{i=0}^{j-1} \log (\frac{1}{\Delta_i(\mathfrak{A})})}{\log (\frac{1}{d_j(\mathfrak{A})})} \\
   = & n-1 - \limsup_{j \rightarrow \infty} \frac{\sum_{i=0}^{j-1} (n-1)\delta/2 (1+\delta)^{i+1}\log h(\mathfrak{a}_0) + O(i)}{(1+\delta/2)(1+\delta)^j \log h(\mathfrak{a}_0) +O(j)} \\
   = & n-1 - \limsup_{j \rightarrow \infty} \frac{(n-1)(1+\delta)^{j+1}/2 \log h(\mathfrak{a}_0) + O(j^2)}{(1+\delta/2)(1+\delta)^j \log h(\mathfrak{a}_0) + O(j)} \\
   = & n-1 - \frac{(n-1)(1+\delta)}{2 +\delta} = \frac{n-1}{2+\delta}.
  \end{array}
$$
This completes the proof.
\end{proof}
Combining Theorem \ref{Hausdorff_dimension_second_component}, Proposition \ref{Hausdorff_dimension_first_component} and Lemma \ref{slicing_theorem},
we get that 
$$\dim_H \mathfrak{B}_k \geq (k-1)(n-1) + \frac{n-1}{2+\delta},$$
for all $\delta >0$. By letting $\delta \rightarrow 0$, we show that
\begin{equation}
 \label{lower_bound_hausdorff_dim_equation}
 \dim_H \mathfrak{B}_k \geq k(n-1) -\frac{n-1}{2}.
\end{equation}
\subsection{Upper bound of the Hausdorff dimension}
The basi idea to get the upper bound of the Hausdorff dimension of $\dim_H \mathfrak{B}_k$ is the following: 
\par At first we choose a small constant $\rho >0$, and define 
$$E(\rho) = \left\{\mathcal{V}_k \in (\R^{n-1})^k: W(\mathcal{V}_k, t) \geq \rho^{-1} 
\quad \text{ for all } t>0 \text{ large enough } \right\},$$
then obviously $\mathfrak{B}_k \subset E(\rho)$. We then construct an indexed self-similar covering
$(\mathcal{B}, J, \varrho)$ of $E(\rho)$, which is defined as follows:
\begin{definition} 
\label{def_self_similar_covering}
Let $\mathcal{B}$ be a countable covering of a subset $E\subset \mathbb{R}^{l}$ by
bounded subsets of $\mathbb{R}^l$ and assume that it is indexed by some countable
set $J$; let $\varrho$ be a function from the set $J$ to the set of all nonempty
subsets of $J$. For any $\alpha \in J$ we write $B(\alpha)$ for the element of $\mathcal{B}$
indexed by $\alpha$. We say $(\mathcal{B}, J, \varrho)$ is an indexed self-similar covering of
$E$ (the indexing function $\iota : J \rightarrow \mathcal{B}$ being implicit) if there exists a
$\lambda$, $0 < \lambda < 1$ such that for every $x \in E$ we have a sequence $(\alpha_j )$ of
elements in $J$ satisfying
\begin{enumerate}
\item $\cap B(\alpha_j) =\{x\}$,
\item $\mathrm{diam} B(\alpha_{j+1}) < \lambda \mathrm{diam} B(\alpha_j )$ for all $j$, and
\item $\alpha_{j+1} \in \varrho(\alpha_j)$ for all $j$.
\end{enumerate}
\end{definition}
And then we apply the following theorem of Cheung to get the upper bound of $\dim_H E(\rho)$:
\begin{theorem}{(See ~\cite[Theorem 5.3]{Cheung1})}
\label{tool_for_upper_bound_hausdorff_dimension}
Let $(\mathcal{B}, J, \varrho)$ be an indexed self-similar covering of a
subset $E \subset \mathbb{R}^l$ and suppose there is an $s > 0$ such that for every
$\alpha \in J$
$$ \sum_{\alpha' \in \varrho(\alpha)} (\mathrm{diam} B(\alpha'))^s \leq (\mathrm{diam} B(\alpha))^s$$
Then $\dim_H E \leq s$.
\end{theorem}
So our first step is to construct the indexed self-similar covering $(\mathcal{B}, J, \varrho)$ of $E(\rho)$.
\par The construction of $(\mathcal{B}, J, \varrho)$ basically follows 
from the work of Cheung (see ~\cite{Cheung1}), with some minor modification.
\par We need some preparation before the construction.

\begin{lemma}
 \label{at_least_one_cusp}
For any noncompact hyperbolic space $\Gamma\setminus \mathbb{H}^n$ with finite volume,
there exists a constant $c>0$ such that for all $X>0$ large enough, and any closed ball $B\in \mathbb{R}^{n-1}$ of
radius $\frac{c}{\sqrt{X}}$, there is at least one cusp of $\Gamma\setminus \mathbb{H}^n$ inside $B$ with height less than
or equal to $X$.
\end{lemma}
\begin{proof}
 Let $\mathbf{x} \in \mathbb{R}^{n-1}$ denote the center of the ball, and consider the rough
 spectrum of $\mathcal{G}_{\mathbf{x}}$, say
 $\{\mathfrak{a}_i:i\in \mathbb{N} \}$. Let $\mathfrak{a}_i$ denote the cusp 
 with largest height less than or equal to $X$, then we have 
 $$\|\mathbf{x} - \mathfrak{a}_i\| \asymp \frac{1}{\sqrt{h(\mathfrak{a}_i) h(\mathfrak{a}_{i+1})}} \ll \frac{1}{\sqrt{X}}$$
 since $h(\mathfrak{a}_{i+1}) \geq X$ and $h(\mathfrak{a}_i) \gg 1$. 
 Therefore there exists some constant $c > 0$ such that 
 $\mathfrak{a}_i$ is inside the ball centered at $\mathbf{x}$ with radius $\frac{c}{\sqrt{X}}$. 
 \par This proves the lemma.
\end{proof}
\begin{remark} For each component $\Gamma_i \setminus \mathbb{H}^n$ and each positive integer $N$, 
then the above lemma tells that there 
exists a countable subset $\mathfrak{E}(i, N)$ of cusps of $\Gamma_i$ such that for the constant $c$ and some smaller costant $c'$
such that every closed ball $B \subset \mathbb{R}^{n-1}$ of radius $\frac{c}{\sqrt{N}}$ contains at least one element of $\mathfrak{E}(i,N)$, and 
in every closed ball of radius $\frac{c'}{\sqrt{N}}$ there is at most one element of $\mathfrak{E}(i,N)$. And moreover, every element 
in $\mathfrak{E}(i,N)$ has height less than or equal to $N$. We fix these subsets.
\end{remark}
\begin{definition}
\label{construction_of_self_similar_covering}
\par Let $\mathcal{Q}_i \subset \mathbb{R}^{n-1}$ denote the set of cusps of $\Gamma_i$, and $I =\{1,\dots , k\}$.
Define $J \subset \mathcal{Q}_1\times \cdots \times \mathcal{Q}_k \times I\times I$ to be the collection of elements
$(\mathfrak{a}_1, \dots, \mathfrak{a}_k , i,j)$ satisfying the following conditions:
\begin{enumerate}
 \item $h(\mathfrak{a}_j) < \rho h(\mathfrak{a}_i)$ 
 \item $h(\mathfrak{a}_l) \in \mathfrak{E}(l,\lfloor h(\mathfrak{a}_i) h(\mathfrak{a}_j)\rfloor)$ for other index $l$.
\end{enumerate}
For all such $(\mathfrak{a_1},\dots, \mathfrak{a}_k, i,j)$, we denote by 
$B(\mathfrak{a}_1,\dots, \mathfrak{a}_k, i,j) \subset \mathbb{R}^{n-1}\times \cdots, \times \mathbb{R}^{n-1}$ the ball centered at
$(\mathfrak{a}_1,\dots, \mathfrak{a}_k)$ with radius $\frac{c}{\sqrt{h(\mathfrak{a}_i) h(\mathfrak{a}_j)}}$ for some constant $c$,
with respect to the supreme norm
of the norm in $\mathbb{R}^{n-1}$, i.e.,
$$\|(\mathbf{x}_1,\dots, \mathbf{x}_k)\| = \max_{1\leq i\leq k} \|\mathbf{x}_i\|$$
We define $\mathcal{B}$ to be the collection of all $B(\mathfrak{a}_1,\dots, \mathfrak{a}_k, i,j)$ for 
$(\mathfrak{a}_1,\dots, \mathfrak{a}_k,i,j) \in J$.
And we define 
$$A(\mathfrak{a}_1,\dots, \mathfrak{a}_k, i,j) \subset B(\mathfrak{a}_1,\dots, \mathfrak{a}_k, i,j)$$ 
to be the 
neighborhood of $(\mathfrak{a}_1,\dots, \mathfrak{a}_k)$ whose $i$th component has radius $\frac{c}{h(\mathfrak{a}_i)}$ and other
components have radius $\frac{c}{\sqrt{h(\mathfrak{a}_i) h(\mathfrak{a}_j)}}$.
\par We define $\varrho$ as follows: for $(\mathfrak{a}_1,\dots, \mathfrak{a}_k,i,j) \in J$, we define 
$\varrho(\mathfrak{a}_1,\dots, \mathfrak{a}_k,i,j) \subset J$ to be the collection of elements 
$(\mathfrak{a}'_1, \dots, \mathfrak{a}'_k, j , j') \in J$ satisfying the following conditions:
\begin{enumerate}
 \item $\mathfrak{a}'_j$ is a successive cusp point of $\mathfrak{a}_j$, 
 we denote this condition by $\mathfrak{a}_j \mapsto \mathfrak{a}'_j$.
 \item if $j' = i$, then $h(\mathfrak{a}'_i) > h(\mathfrak{a}_i)$ and $\|\mathfrak{a}'_i -\mathfrak{a}_i \| \leq \frac{c}{h(\mathfrak{a}_i)}$
 \item $h(\mathfrak{a}_i) < h(\mathfrak{a}'_j)$ and $h(\mathfrak{a}_j) < h(\mathfrak{a}'_{j'})$.
 \item $A(\mathfrak{a}_1,\dots,\mathfrak{a}_k,i,j)\cap A(\mathfrak{a}'_1,\dots, \mathfrak{a}'_k, j,j') \neq \emptyset$
 \item $h^2(\mathfrak{a}_i) \leq h(\mathfrak{a}'_j) h(\mathfrak{a}'_{j'})$.
\end{enumerate}
\end{definition}
We will prove that the above construction gives an indexed self-similar covering of $E(\sigma)$:
\begin{proposition}
 Definition \ref{construction_of_self_similar_covering} gives an indexed self-similar covering 
 of $E(\sigma)$.
\end{proposition}
\begin{proof}
 For any $\mathcal{V}_k = (\mathbf{x}_1,\dots, \mathbf{x}_2) \in E(\rho)$, then we have 
the local minima of $W(\mathcal{V}_k, t)$ is always greater than $\rho^{-1}$ for $t$ large enough. 
Then there exists 
a sequence of times $\{t_p: p \in \mathbb{N}\}$ such that at each
 $t = t_p$ it admits a local minimum. Then from the previous argument we have at this moment, 
 $W(\mathcal{V}_k, t)$ changes
from $W_{i(p)}(\mathbf{x}_{i(p)}, t)$ to $W_{i(p+1)}(\mathbf{x}_{i(p+1)},t)$ for some indices 
$i(p),i(p+1) \in I$. And in this case if we denote 
by $\mathfrak{a}(p)_{i(p)}$ and 
$\mathfrak{a}(p+1)_{i(p+1)}$ the corresponding cusps of $\mathcal{G}_{\mathbf{x}_{i(p)}}$ 
and $\mathcal{G}_{\mathbf{x}_{i(p+1)}}$ respectively, at this moment, and denote
by $\mathfrak{b}(p)_{i(p)}$ and $\mathfrak{b}(p+1)_{i(p+1)}$ the next cusps of $\mathcal{G}_{\mathbf{x}_{i(p)}}$ and $\mathcal{G}_{\mathbf{x}_{i(p+1)}}$ respectively,
then the previous argument tells that
$$W(\mathcal{V}_k, t_p) =  \frac{1}{\sqrt{h(\mathfrak{a}(p)_{i(p)})h(\mathfrak{a}(p+1)_{i(p+1)})}\|\x_{i(p)} - \mathfrak{a}(p)_{i(p)}\|},$$ and 
$$e^{t_p} \asymp \frac{\sqrt{h(\mathfrak{a}(p+1)_{i(p+1)})}}{\sqrt{h(\mathfrak{a}(p)_{i(p)})}\|\x_{i(p)} - \mathfrak{a}(p)_{i(p)}\|}.$$
Note that if we denote by $\mathfrak{b}(p)_{i(p)}$ the next cusp $\mathcal{G}_{\x_{i(p)}}$ roughly enters after leaving $\mathfrak{a}(p)_{i(p)}$, 
we have that (see Proposition \ref{distance_between_point_and_rough_cusp}) 
$$\|\x_{i(p)} - \mathfrak{a}(p)_{i(p)}\| \asymp \frac{1}{\sqrt{h(\mathfrak{a}(p)_{i(p)}) h(\mathfrak{b}(p)_{i(p)})}}.$$
Then $W(\mathcal{V}_k, t_p) > \rho^{-1}$ implies that $h(\mathfrak{a}(p+1)_{i(p+1)}) < \rho^2 h(\mathfrak{b}(p)_{i(p)})$. 
Then we have a sequence of triplets 
$$\{(\mathfrak{a}(p)_{i(p)}, \mathfrak{b}(p)_{i(p)}, i(p)) : p\in \mathbb{N}\}$$
It is easy to see that $h(\mathfrak{a}(p)_{i(p)})$ and $h(\mathfrak{b}(p)_{i(p)})$ are both increasing with respect to 
$p$.
\par Next we define a subsequence $\{(\mathfrak{a}(p_l)_{i(p_l)}, \mathfrak{b}(p_l)_{i(p_l)}, i(p_l)): l\in \mathbb{N}\}$
of $\{(\mathfrak{a}(p)_{i(p)}, \mathfrak{b}(p)_{i(p)}, i(p))\}$ as follows: we start with some 
large $p_0\in \mathbb{N}$ and suppose $p_l$ is defined, we define $p_{l+1}$ to be the smallest subindex $p$ such that 
$$h^2(\mathfrak{b}(p_l)_{i(p_l)}) \leq h(\mathfrak{b}(p)_{i(p)}) h(\mathfrak{a}(p+1)_{i(p+1)})$$
This defines the subsequence $\{(\mathfrak{a}(p_l)_{i(p_l)}, \mathfrak{b}(p_l)_{i(p_l)}, i(p_l) ): l\in \mathbb{N}\}$.
\par From the definition it is easy to see that 
$$h^2(\mathfrak{b}(p_l)_{i(p_l)}) \leq h(\mathfrak{b}(p_{l+1})_{i(p_{l+1})}) h(\mathfrak{a}(p_{l+2})_{i(p_{l+2})})$$
for all $l\in \mathbb{N}$.
\par Moreover, from the definition we have 
$$h^2(\mathfrak{b}(p_l)_{i(p_l)}) > h(\mathfrak{b}(p_{l+1} -1)_{i(p_{l+1}-1)})h(\mathfrak{a}(p_{l+1})_{i(p_{l+1})}) > \frac{1}{\rho^2} h^2(\mathfrak{a}(p_{l+1})_{i(p_{l+1})})$$
this implies that 
$$h(\mathfrak{a}(p_{l+1})_{i(p_{l+1})}) < \rho h(\mathfrak{b}(p_l)_{i(p_l)})$$
For simplicity, in the following argument, we write $l$ for $p_l$. Then we have 
$$h^2(\mathfrak{b}(l)_{i(l)}) \leq h(\mathfrak{b}(l+1)_{i(l+1)}) h(\mathfrak{a}(l+2)_{i(l+2)})$$ 
and 
$$h(\mathfrak{a}(l+1)_{i(l+1)}) < \rho h(\mathfrak{b}(l)_{i(l)})$$
Then we consider 
$$\mathfrak{u}(l) = (\mathfrak{c}(l)_1,\dots, \mathfrak{c}(l)_k, i(l), i(l+1))$$
for each $l\in \mathbb{N}$ such that 
\begin{enumerate}
 \item $\mathfrak{c}(l)_{i(l)} = \mathfrak{b}(l)_{i(l)}$
 \item $\mathfrak{c}(l)_{i(l+1)} = \mathfrak{a}(l+1)_{i(l+1)}$
 \item for other index $j$, we choose $\mathfrak{c}(l)_j$ to be a cusp in 
 $\mathfrak{E}(j, \lfloor h(\mathfrak{c}(l)_{i(l)})h(\mathfrak{c}(l)_{i(l+1)})\rfloor)$ such that 
 $\|\mathbf{x}_j - \mathfrak{c}(l)_j\| \leq \frac{c}{\sqrt{h(\mathfrak{c}(l)_{i(l)}) h(\mathfrak{c}(l)_{i(l+1)})}}$, this can always be done
 because of the property of $\mathfrak{E}(i, N)$ (see the remark after Lemma \ref{at_least_one_cusp}).
\end{enumerate}
It can be seen from the above argument that $\mathfrak{u}(l) \in J$. Moreover, we claim that 
$$\mathcal{V}_k \in B(\mathfrak{c}(l)_1,\dots, \mathfrak{c}(l)_k , i(l), i(l+1)).$$
At first, $\mathfrak{a}(l+1)_{i(l+1)}$ and $\mathfrak{b}(l+1)_{i(l+1)}$ are two consecutive cusps in the rough spectrum of $\x_{i(l+1)}$,
so we have 
$$\|\mathbf{x}_{i(l+1)} - \mathfrak{a}(l+1)_{i(l+1)}\| \asymp 
\frac{1}{\sqrt{h(\mathfrak{a}(l+1)_{i(l+1)})h(\mathfrak{b}(l+1)_{i(l+1)})}} \leq 
\frac{1}{\sqrt{h(\mathfrak{a}(l+1)_{i(l+1)})h(\mathfrak{b}(l)_{i(l)})}}, $$
and 
$$\|\mathbf{x}_{i(l)} -\mathfrak{b}(l)_{i(l)}\| \ll 
\frac{1}{h(\mathfrak{b}(l)_{i(l)})} \leq \frac{1}{\sqrt{h(\mathfrak{a}(l+1)_{i(l+1)})h(\mathfrak{b}(l)_{i(l)})}}.$$
For any other index $j$, we have 
$$\|\mathbf{x}_j - \mathfrak{c}(l)_j\| \leq \frac{c}{\sqrt{h(\mathfrak{c}(l)_{i(l)}) h(\mathfrak{c}(l)_{i(l+1)})}},$$
from our choice of $\mathfrak{c}(l)_j$. Therefore, we may choose appropriate constant $c>0$ in the definition of 
$B(\mathfrak{c}(l)_1,\dots, \mathfrak{c}(l)_k, i(l), i(l+1))$ such that
$$\mathcal{V}_k \in B(\mathfrak{c}(l)_1,\dots, \mathfrak{c}(l)_k, i(l), i(l+1)),$$ for all $l \in \mathbb{N}$.
Moreover, from the argument above we have 
$$\mathcal{V}_k \in A(\mathfrak{c}(l)_1,\dots, \mathfrak{c}(l)_k, i(l), i(l+1)).$$
\par The diameter of $B(\mathfrak{u}(l))$
$$\mathrm{diam} B(\mathfrak{u}(l)) = \frac{2c}{\sqrt{h(\mathfrak{a}(l+1)_{i(l+1)}) h(\mathfrak{b}(l)_{i(l)})}}.$$
Therefore 
$$
\begin{array}{cl}
& \frac{\mathrm{diam} B(\mathfrak{u}(l+1))}{ \mathrm{diam} B(\mathfrak{u}(l))} \\
= & 
\left(\frac{h(\mathfrak{a}(l+1)_{i(l+1)}) h(\mathfrak{b}(l)_{i(l)})}{h(\mathfrak{a}(l+2)_{i(l+2)}) h(\mathfrak{b}(l+1)_{i(l+1)})} 
\right)^{1/2} \\
\leq & \left( \frac{h(\mathfrak{a}(l+1)_{i(l+1)}) h(\mathfrak{b}(l)_{i(l)})}{h^2(\mathfrak{b}(l)_{i(l)})}\right)^{1/2} \\
\leq & \rho^{1/2}.
\end{array}
$$
This shows that $\mathrm{diam} B(\mathfrak{u}(l+1)) \leq \lambda \mathrm{diam} B(\mathfrak{u}(l))$
for $\lambda =\rho^{1/2} < 1$. And moreover this implies that $\mathrm{diam} B(\mathfrak{u}(l)) \rightarrow 0$ as 
$l\rightarrow \infty$. Combining this with $\mathcal{V}_k \in B(\mathfrak{u}(l))$ for any $l\in \mathbb{N}$,
we have that 
$$\bigcap_{l=0}^{\infty} B(\mathfrak{u}(l)) =\{\mathcal{V}_k\}$$
To show that $(\mathcal{B},J,\varrho)$ is a self-similar covering of $E(\rho)$, the last thing is to verify that
$$(\mathfrak{c}(l+1)_1,\dots, \mathfrak{c}(l+1)_k, i(l+1), i(l+2))
\in \varrho(\mathfrak{c}(l)_1,\dots, \mathfrak{c}(l)_k, i(l), i(l+1)).$$
At first, $\mathfrak{a}(l+1)_{i(l+1)} \mapsto \mathfrak{b}(l+1)_{i(l+1)}$ since they are consecutive cusps in
the rough spectrum of $\x_{i(l+1)}$ (see Proposition \ref{distance_between_point_and_rough_cusp}), so the first condition is verified.
\par If $i(l+2) = i(l)$, then $\mathfrak{c}(l)_{i(l)}$ and $\mathfrak{c}(l+1)_{i(l)}$ are both in the rough spectrum of 
$\mathbf{x}_{i(l)}$, this ensures the second condition.
\par The third condition is true since $h(\mathfrak{a}(l)_{i(l)})$ and$h( \mathfrak{b}(l)_{i(l)})$ are both increasing 
with respect to $l$. The fourth condition is true since $\mathcal{V}_k \in A(\mathfrak{u}(l))
\cap A(\mathfrak{u}(l+1))$ for all $l\in\mathbb{N}$. The last condition is directly from the definition of 
$(\mathfrak{a}(l)_{i(l)}, \mathfrak{b}(l)_{i(l)}, i(l))$. Thus we prove that 
$$(\mathfrak{c}(l+1)_1,\dots, \mathfrak{c}(l+1)_k, i(l+1), i(l+2))
\in \varrho(\mathfrak{c}(l)_1,\dots, \mathfrak{c}(l)_k, i(l), i(l+1)).$$
\par This shows that $(\mathcal{B}, J, \varrho)$ is a self-similar covering of $E(\rho)$.
\end{proof}
\par Now we are ready to apply Theorem \ref{tool_for_upper_bound_hausdorff_dimension} to give the upper bound of 
the Hausdorff dimension of $E(\rho)$.
\par We will need the following lemma concerning the upper bound of the number of cusp points inside a ball $B \subset \R^{n-1}$ 
with height bounded by $X >0$:
\begin{lemma}
 \label{cusp_counting_upper_bound}
 For a noncompact hyperbolic space $\Gamma\setminus \mathbb{H}^n$ 
 and a closed ball $B \subset \mathbb{R}^{n-1}$, we denote by $F_B(t)$ the number of cusps of $\Gamma$ inside $B$ with height 
 less than or equal to $t$ for $t > 0$. Then there exists some constant $\omega >0$ such that for every $X >0$ large enough, we have 
 $$\int_{1}^X \frac{1}{t^{(n-1)/2}} d F_B(t) \leq \omega X^{(n-1)/2} \mathrm{Vol}(B).$$
\end{lemma}
\begin{proof}
 We may choose $X$ large enough such that $\frac{1}{\sqrt{X}} \ll \mathrm{diam}(B)$.
  Note that the left hand hand of the inequality above is equal to
  $$\sum_{\mathfrak{a}} \frac{1}{h^{(n-1)/2}(\mathfrak{a})},$$
  where $\mathfrak{a}$ runs over all cusps of $\Gamma$ inside $B$ with height less than or equal to $X$ (we may assume that every cusp
  has height greater or equal to $1$ without loss of generality since the height is uniformly bounded from $0$). 
  For every two cusps $\mathfrak{a}_1, \mathfrak{a}_2 \in B$, both with height less than or equal to $X$, we will have 
  $$\|\mathfrak{a}_1 -\mathfrak{a}_2 \| \gg \frac{1}{\sqrt{h(\mathfrak{a}_1) h(\mathfrak{a}_2)}} \geq \frac{1}{\sqrt{X h(\mathfrak{a}_i)}},$$
  for $i= 1,2$. Therefore there exists some constant $c$ such that if we denote by $B(\mathfrak{a})$ a ball centered at $\mathfrak{a}$ with radius
  $\frac{c}{\sqrt{X h(\mathfrak{a})}}$, for every cusp $\mathfrak{a}$ inside $B$ with height less than or equal to $X$, then we have every ball is 
  disjoint from others. And since we choose $X$ large enough, we have 
  $$\bigcup_{\mathfrak{a}} B(\mathfrak{a}) \subset B',$$
  where $B'$ is some larger ball sharing the center of $B$ and has diameter $\tilde{c}\mathrm{diam}(B)$ for some other constant $\tilde{c}$. This implies that
  $$\sum_{\mathfrak{a}} \mathrm{Vol}(B(\mathfrak{a})) \leq \mathrm{Vol}(B') = \tilde{c}^{n-1} \mathrm{Vol} (B).$$
  This is equivalent to 
  $$\sum_{\mathfrak{a}} \frac{c^{n-1}}{X^{(n-1)/2} h^{(n-1)/2}(\mathfrak{a})} \leq \zeta \mathrm{Vol}(B)$$
  for some constant $\zeta >0$. The above inequality implies that
  $$\sum_{\mathfrak{a}} \frac{1}{h^{(n-1)/2}(\mathfrak{a})} \leq \omega X^{(n-1)/2} \mathrm{Vol}(B),$$
 and this implies the conclusion immediately.
 \end{proof}
 Now let us apply Theorem \ref{tool_for_upper_bound_hausdorff_dimension} to give the upper bound of the 
 Hausdorff dimension of $E(\rho)$:
 \begin{theorem}
  \label{upper_bound_hausdorff_dimension}
  For $\rho >0$ small enough, we have 
 $$\dim_H E(\rho) \leq (k-1)(n-1) +(n-1)/2 +\rho^{1/8}.$$
 \end{theorem}
\begin{proof}
  For a fixed $(\mathfrak{a}_1,\dots, \mathfrak{a}_k, i, j)\in J$, and $s = k(n-1) - (n-1)/2+\rho^{1/8}$, we want to estimate
$$
\sum_{(\mathfrak{a}'_1,\dots, \mathfrak{a}'_k, j, j') \in \varrho(\mathfrak{a}_1,\dots, \mathfrak{a}_k, i, j)} 
\left(\frac{\mathrm{diam} B(\mathfrak{a}'_1,\dots, \mathfrak{a}'_k, j, j')}
{\mathrm{diam} B(\mathfrak{a}_1,\dots, \mathfrak{a}_k, i, j)} \right)^s .
$$
Let us denote 
$$\varrho_l(\mathfrak{a}_1,\dots, \mathfrak{a}_k, i, j) 
=\{(\mathfrak{a}'_1,\dots, \mathfrak{a}'_k, j, l)\in \varrho(\mathfrak{a}_1,\dots, \mathfrak{a}_k, i, j)\}.$$
Then the above summation can be separated as 
$$\sum_{l} \sum_{\varrho_l(\mathfrak{a}_1,\dots, \mathfrak{a}_k, i, j)} \left(\frac{h(\mathfrak{a}_i)h(\mathfrak{a}_j)}
{h(\mathfrak{a}'_j) h(\mathfrak{a}'_l)}\right)^{s/2}$$
\par Let us put $\frac{h(\mathfrak{a}'_j)}{h(\mathfrak{a}_j)} = a$ and $\frac{h(\mathfrak{a}'_l)}{h(\mathfrak{a}_i)} =b$, then we have 
$a > \rho^{-1}$ and $b < \rho a$.
\par We separate the choice of $l$ into two case:
\begin{enumerate}
 \item $l=i$: let us denote by $H(a)$ the number of cusps $\mathfrak{a}'_j$ such that $\mathfrak{a}_j 
 \mapsto \mathfrak{a}_j$ and $h(\mathfrak{a}'_j) \leq h(\mathfrak{a}_j)$, then it is easily seen that 
 $H(a) \asymp a^{(n-1)/2}$ (suppose $\mathfrak{a}_j =\gamma \xi \infty$, then 
 $\mathfrak{a}'_j = \gamma \xi u(\mathbf{n}) \xi^{-1} s \xi' \infty$, choice of $s$ and $\xi'$ are both finite,
 and $\|\mathbf{n}\|^2 \leq a$ implies the number of choices for $\mathbf{n}$ is $\asymp a^{(n-1)/2}$). And since
 $$\|\mathfrak{a}'_i -\mathfrak{a}_i\| \leq \frac{c}{h(\mathfrak{a}_i)}$$ 
 We denote by $F(b)$ the number of cusps in side the ball $B\left(\mathfrak{a}_i, \frac{c}{h(\mathfrak{a}_i)}\right)$ 
 with height less than or equal to $b h(\mathfrak{a}_i)$. And once 
 we fix $a$ and $b$, then for any other component $w$, cusps are elements in $\mathfrak{E}(w, \lfloor 
 h(\mathfrak{a}'_j)h(\mathfrak{a}'_i)\rfloor)$ inside the ball centered at $\mathfrak{a}_w$ with 
 radius $\frac{c}{\sqrt{h(\mathfrak{a}_i)h(\mathfrak{a}_j)}}$, from the definition of $\mathfrak{E}(w, N)$, we have the number of 
 choices for $\mathfrak{a}'_w$ is 
 $$\asymp \left(\frac{h(\mathfrak{a}'_i) h(\mathfrak{a}'_j)}{h(\mathfrak{a}_i)h(\mathfrak{a}_j)}\right)^{(n-1)/2} =(ab)^{(n-1)/2}.$$
 Therefore the summation
 $$\sum_{\varrho_i(\mathfrak{a}_1,\dots, \mathfrak{a}_k, i,j)} \left(
 \frac{h(\mathfrak{a}_i) h(\mathfrak{a}_j)}{h(\mathfrak{a}'_i) h(\mathfrak{a}'_j)}\right)^{s/2}$$
 can be estimated as the following summation
 $$\int_{a > \rho^{-1}}\int_{1}^{\rho a} (ab)^{(k-2)(n-1)/2} \frac{1}{(ab)^{s/2}} d F(b) d H(a).$$
 By simplifying it we have 
 $$\int_{a> \rho^{-1}} a^{(k-2)(n-1)/2 - s/2} \int_{1}^{\rho a} b^{(k-1)(n-1)/2 -s/2} \frac{1}{b^{(n-1)/2}} d F(b) d H(a).$$
 Let $G(x) = \int_{1}^{x} \frac{1}{b^{(n-1)/2}} d F(b)$, then from Lemma \ref{cusp_counting_upper_bound}, we have 
 $$G(x) \ll x^{(n-1)/2}.$$
 This is true since in $B(\mathfrak{a}_i, c)$, $\mathfrak{a}_i$ is the cusp with smallest height.
 \par Then the above integral equals
 $$
 \begin{array}{cl} & \int_{a > \rho^{-1}} \int_{1}^{\rho a} a^{(k-2)(n-1)/2 -s/2}  b^{(k-1)(n-1)/2 -s/2} d G(b) d H(a) \\
 = & \int_{a > \rho^{-1}} a^{(k-2)(n-1)/2 -s/2}  \left((\rho a)^{(k-1)(n-1)/2 -s/2} G(\rho a) -\int_1^{\rho a} G(b) d b^{(k-1)(n-1)/2 -s/2} \right) d H(a)\\
 = & \int_{a > \rho^{-1}} a^{(k-2)(n-1)/2 -s/2}  \left( (\rho a)^{(k-1)(n-1)/2 -s/2} G(\rho a) \right. \\ 
  & \left. +  (s/2 - (k-1)(n-1)/2) \int_1^{\rho a} b^{(k-1)(n-1)/2 -s/2 -1} G(b) d b \right) d H(a) \\
 \ll & \int_{a > \rho^{-1}} a^{(k-2)(n-1)/2 -s/2} \left[(\rho a)^{(k-1)(n-1)/2 -s/2} (\rho a)^{(n-1)/2} \right. \\ 
 & \left. + (s/2 - (k-1)(n-1)/2)\int_1^{\rho a} b^{(k-1)(n-1)/2 -s/2 -1} b^{(n-1)/2} db\right] d H(a)  \\
 = & \int_{a > \rho^{-1}} a^{(k-2)(n-1)/2 -s/2}  \left((\rho a)^{k(n-1)/2 -s/2} \right. \\
  &  \left. +\frac{s/2 - (k-1)(n-1)/2}{k(n-1)/2 -s/2} (\rho a)^{k(n-1)/2 -s/2}\right) \dd H(a) \\
 = & \frac{\rho^{k(n-1)/2 -s/2} (n-1)}{k(n-1) -s} \int_{a > \rho^{-1}} a^{(k-1)(n-1) -s} \dd H(a) \\
 = & \frac{\rho^{k(n-1)/2 -s/2} (n-1)}{k(n-1) -s} \left[ \left(a^{(k-1)(n-1) -s} H(a)\right)_{\rho^{-1}}^{\infty} 
 - \int_{a > \rho^{-1}} H(a) \dd a^{(k-1)(n-1) -s} \right].
 \end{array}
$$
Since $H(a) \asymp a^{(n-1)/2}$, the above integral asymptotically equals:
$$
\begin{array}{cl}
 & \frac{\rho^{k(n-1)/2 -s/2} (n-1)}{k(n-1) -s} \left[\left(a^{(k-1)(n-1) -s} a^{(n-1)/2}\right)_{\rho^{-1}}^{\infty} \right. \\ 
 & \left. + (s-(k-1)(n-1)) \int_{a > \rho^{-1}} a^{(k-1)(n-1)-s -1 + (n-1)/2} \dd a \right] \\
 = & \frac{\rho^{k(n-1)/2 -s/2} (n-1)}{k(n-1) -s} \left(- (\rho)^{s - (k-1)(n-1) -(n-1)/2} \right. \\ 
 & \left. +\frac{s- (k-1)(n-1)}{s-(k-1)(n-1) -(n-1)/2} (\rho)^{s-(k-1)(n-1) -(n-1)/2}\right) \\
 = & \frac{(n-1)^2}{2[s-(n-1)(k-1)-(n-1)/2][k(n-1)-s]} \rho^{k(n-1)/2 -s/2} \rho^{s-(k-1)(n-1) -(n-1)/2} .
\end{array}
$$
 \item $l\neq i$: let $H(a)$ be as above, and let $F(b)$ denote the number of cusps $\mathfrak{a}'_l$ such that 
 $$\|\mathfrak{a}'_l - \mathfrak{a}_l\| \leq \frac{c}{\sqrt{h(\mathfrak{a}_i) h(\mathfrak{a}_j)}},$$ 
 and $h(\mathfrak{a}'_l) \leq b h(\mathfrak{a}_i)$ Then from Lemma \ref{cusp_counting_upper_bound} tells that 
 $$\int_{b \leq X} \frac{1}{(b h(\mathfrak{a}_i))^{(n-1)/2}} \dd F(b) \ll 
 (X h(\mathfrak{a}_i))^{(n-1)/2} \frac{1}{(h(\mathfrak{a}_i)h(\mathfrak{a}_j))^{(n-1)/2}},$$
 which implies 
 $$\int_{b \leq X} \frac{1}{b^{(n-1)/2}} d F(b) \ll X^{(n-1)/2} \left(\frac{h(\mathfrak{a}_i)}{h(\mathfrak{a}_j)}\right)^{(n-1)/2}.$$
 We denote $G(X) = \int_{b \leq X} \frac{1}{b^{(n-1)/2}} d F(b)$, then we have 
 $$G(X) \ll X^{(n-1)/2} \left(\frac{h(\mathfrak{a}_i)}{h(\mathfrak{a}_j)}\right)^{(n-1)/2}.$$
 Now we fix $a$ and $b$ as above, for subindex $w \neq j, i, l$, $\mathfrak{a}'_w \in \mathfrak{E}(w, \lfloor h(\mathfrak{a}'_j) h(\mathfrak{a}'_l) \rfloor)$ and 
 $$\|\mathfrak{a}'_w -\mathfrak{a}_w\| \leq \frac{c}{\sqrt{h(\mathfrak{a}_i) h(\mathfrak{a}_j)}},$$ the number of choices for $\mathfrak{a}'_w$ is 
 asymptotically equal to 
 $$\left(\frac{h(\mathfrak{a}'_j) h(\mathfrak{a}'_l)}{h(\mathfrak{a}_i) h(\mathfrak{a}_j)}\right)^{(n-1)/2} = (ab)^{(n-1)/2}.$$
 For index $i$ since 
 $$\|\mathfrak{a}'_i - \mathfrak{a}_i\| \leq \frac{c}{h(\mathfrak{a}_i)},$$
 the number of choices for $\mathfrak{a}'_i$ is asymptotically equal to
 $$\left(\frac{h(\mathfrak{a}'_j) h(\mathfrak{a}'_l)}{h^2(\mathfrak{a}_i)}\right)^{(n-1)/2} = (ab)^{(n-1)/2} 
 \left(\frac{h(\mathfrak{a}_j)}{h(\mathfrak{a}_i)}\right)^{(n-1)/2}.$$
 Thus the summation
 $$\sum_{\varrho_l(\mathfrak{a}_1, \dots, \mathfrak{a}_k, i,j)} \left(\frac{h(\mathfrak{a}_i)h(\mathfrak{a}_j)}
 {h(\mathfrak{a}'_j) h(\mathfrak{a}'_l)}\right)^{s/2}$$
 can be estimated as
 $$
 \begin{array}{cl}
  & \int_{a > \rho^{-1}}  \int_{0}^{\rho a} \left(\frac{h(\mathfrak{a}_j)}{h(\mathfrak{a}_i)}\right)^{(n-1)/2} (ab)^{(k-2)(n-1)/2 -s/2} \dd F(b) \dd H(a) \\
  = & \left(\frac{h(\mathfrak{a}_j)}{h(\mathfrak{a}_i)}\right)^{(n-1)/2} \int_{a > \rho^{-1}} a^{(k-2)(n-1)/2 -s/2} \int_{0}^{\rho a} b^{(k-2)(n-1)/2 - s/2} \dd F(b) \dd H(a) \\
  = & \left(\frac{h(\mathfrak{a}_j)}{h(\mathfrak{a}_i)}\right)^{(n-1)/2} \int_{a > \rho^{-1}} a^{(k-2)(n-1)/2 -s/2} \int_0^{\rho a} b^{(k-1)(n-1)/2 -s/2} \dd G(b) \dd H(a) \\
  = & \left(\frac{h(\mathfrak{a}_j)}{h(\mathfrak{a}_i)}\right)^{(n-1)/2} \int_{a > \rho^{-1}} a^{(k-2)(n-1)/2 -s/2} \left[(\rho a)^{(k-1)(n-1)/2 -s/2} G(\rho a) \right.
  \\ & \left. + (s/2 -(k-1)(n-1)/2)\int_{0}^{\rho a} G(b) b^{(k-1)(n-1)/2 -s/2 -1} \dd b\right] \dd H(a) \\
  \ll & \left(\frac{h(\mathfrak{a}_j)}{h(\mathfrak{a}_i)}\right)^{(n-1)/2} \int_{a > \rho^{-1}} a^{(k-2)(n-1)/2 -s/2}  
  \left[\left(\frac{h(\mathfrak{a}_i)}{h(\mathfrak{a}_j)}\right)^{(n-1)/2}(\rho a)^{k(n-1)/2 -s/2} \right. \\ 
   & \left. + \left(\frac{h(\mathfrak{a}_i)}{h(\mathfrak{a}_j)}\right)^{(n-1)/2} (s/2 -(k-1)(n-1)/2)\int_0^{\rho a} b^{k(n-1)/2 -s/2 -1} \dd b \right] \dd H(a) \\
   = & \frac{\rho^{k(n-1)/2 -s/2} (n-1)}{k(n-1) -s} \int_{a > \rho^{-1}} a^{(k-1)(n-1) -s} \dd H(a) \\
   = & \frac{\rho^{k(n-1)/2 -s/2} (n-1)}{k(n-1) -s} \left[\left(a^{(k-1)(n-1) -s} H(a)\right)_{\rho^{-1}}^{\infty} \right. \\
    & \left. + (s -(k-1)(n-1))\int_{a > \rho^{-1}} H(a) a^{(k-1)(n-1) -s-1} \dd a \right] \\
    \asymp & \frac{(n-1)^2}{2[s-(n-1)(k-1)-(n-1)/2][k(n-1)-s]} \rho^{k(n-1)/2 -s/2} \rho^{s-(k-1)(n-1) -(n-1)/2} .
 \end{array}
$$
We omit several steps in the last estimate since it is the same as the first case.
\end{enumerate}
Summing up the two cases above, we have the summation
$$
\begin{array}{cl}
& \sum_{(\mathfrak{a}'_1,\dots, \mathfrak{a}'_k, j, j') \in \varrho(\mathfrak{a}_1,\dots, \mathfrak{a}_k, i, j)} 
\left(\frac{\mathrm{diam} B(\mathfrak{a}'_1,\dots, \mathfrak{a}'_k, j, j')}
{\mathrm{diam} B(\mathfrak{a}_1,\dots, \mathfrak{a}_k, i, j)} \right)^s  \\ 
\asymp & \frac{(n-1)^2}{2[s-(n-1)(k-1)-(n-1)/2][k(n-1)-s]} \rho^{k(n-1)/2 -s/2} \rho^{s-(k-1)(n-1) -(n-1)/2}. 
\end{array}$$
For any $\rho >0$ small, and $s = (k-1)(n-1) +(n-1)/2 + \rho^{1/8}$, we have:
$$
\begin{array}{cl}
 & \frac{(n-1)^2}{2[s-(n-1)(k-1)-(n-1)/2][k(n-1)-s]} \rho^{k(n-1)/2 -s/2} \rho^{s-(k-1)(n-1) -(n-1)/2} \\
 \asymp & \rho^{(n-1)/4 + \rho^{1/8}/2 -1/8 } \\ 
 \leq & \rho^{1/8}.
\end{array}
$$
Thus for $\rho >0$ small enough, we prove that 
$$\sum_{(\mathfrak{a}'_1,\dots, \mathfrak{a}'_k, j, j') \in \varrho(\mathfrak{a}_1,\dots, \mathfrak{a}_k, i, j)} 
\left(\frac{\mathrm{diam} B(\mathfrak{a}'_1,\dots, \mathfrak{a}'_k, j, j')}
{\mathrm{diam} B(\mathfrak{a}_1,\dots, \mathfrak{a}_k, i, j)} \right)^s \leq 1,$$
for $s = (k-1)(n-1)+ (n-1)/2 +\rho^{1/8}$.
From Theorem \ref{tool_for_upper_bound_hausdorff_dimension}, we show that 
$$\dim_H E(\rho) \leq (k-1)(n-1)+ (n-1)/2 +\rho^{1/8} .$$
\end{proof}
The above theorem shows that 
$$\dim_H \mathfrak{B}_k \leq (k-1)(n-1)+ (n-1)/2 +\rho^{1/8},$$
for all $\rho >0$ small enough. By letting $\rho \rightarrow 0$, we have:
\begin{equation}
 \label{upper_bound_hausdorff_dimension_equation}
 \dim_H \mathfrak{B}_k \leq (k-1)(n-1)+ (n-1)/2.
\end{equation}
\begin{proof}[Proof of Theorem \ref{goal}]
\par Combining (\ref{lower_bound_hausdorff_dim_equation}) and (\ref{upper_bound_hausdorff_dimension_equation}), we complete the proof
of Proposition \ref{main result}.  This concludes Theorem \ref{goal} from the reduction argument in Section 2.
\end{proof}

\bibliography{reference}{}
\bibliographystyle{plain}

\end{document}